\documentclass[a4paper,12pt]{amsart}
\usepackage[utf8]{inputenc}
\usepackage{etex}
\usepackage{lmodern}
\usepackage[english]{babel}
\usepackage{a4wide}

\usepackage{amsmath,amssymb,mathrsfs,amsfonts,dsfont,amsthm}

\usepackage[pdfstartview=FitH,breaklinks=true,bookmarksopen=true]{hyperref}
\hypersetup{colorlinks=true,citecolor=blue,linkcolor=blue}
\usepackage{bookmark}
\usepackage{amsrefs}
\usepackage{amsaddr}

\allowdisplaybreaks

\newtheorem{theorem}{Theorem}[section]

\newtheorem{lemma}[theorem]{Lemma}
\newtheorem{proposition}[theorem]{Proposition}

\theoremstyle{definition}

\newtheorem{remark}[theorem]{Remark}
\newtheorem{assumption}[theorem]{Assumption}
\newtheorem{example}[theorem]{Example}

\newcommand{\Xspace}{\mathbb{X}}
\newcommand{\Yspace}{\mathbb{Y}}
\newcommand{\Zspace}{\mathbb{Z}}
\newcommand{\nlOp}{T}
\newcommand{\sol}{f}
\newcommand{\xdag}{\sol^{\dagger}}
\newcommand{\data}{g}
\newcommand{\gobs}{\data^{\rm obs}}
\newcommand{\lsp}{\left\langle}
\newcommand{\rsp}{\right\rangle}
\newcommand{\projparam}{r}
\newcommand{\projindexset}{\mathcal{J}}
\newcommand{\Cq}{C_1}
\newcommand{\Cdiaglow}{C_2}
\newcommand{\Cdiag}{C_3}

\newcommand{\rbra}[1]{\ensuremath{\left( #1 \right)}}
\newcommand{\sbra}[1]{\ensuremath{\left[ #1 \right]}}
\newcommand{\cbra}[1]{\ensuremath{\left\{ #1 \right\}}}

\newcommand{\absval}[1]{\ensuremath{\left\lvert #1 \right\rvert}}

\newcommand{\norm}[2]{\ensuremath{\left\lVert #1 \right\rVert_{#2}}}
\newcommand{\altnorm}[2]{\ensuremath{\left\lVert #1:#2 \right\rVert}}
\newcommand{\snorm}[2]{\ensuremath{\| #1 \|_{#2}}}
\newcommand{\Tnorm}[2]{\ensuremath{\lVert #1 \rVert_{#2}}}
\newcommand{\pairing}[2]{\ensuremath{\left\langle #1,#2 \right\rangle}}
\newcommand{\diffq}[2]{\ensuremath{\frac{\partial #1}{\partial #2}}}
\newcommand{\phiOp}{\Lambda}

\newcommand{\lmax}[2]{#1 \lor #2}
\newcommand{\lmin}[2]{#1 \land #2}

\newcommand{\dd}{\ensuremath{\mathrm d}}

\newcommand{\Rset}{\ensuremath{\mathds R}}

\newcommand{\Nset}{\ensuremath{\mathds N}}
\newcommand{\Zset}{\ensuremath{\mathds Z}}

\newcommand{\Bs}{\ensuremath{B^s_{2,\infty}(\manifold)}}
\newcommand{\XTkappa}{\ensuremath{\Xspace^T_\kappa}}

\newcommand{\EW}[1]{\mathbf{E}\left[#1\right]}
\newcommand{\TEW}[1]{\mathbf{E}[#1]}

\newcommand{\manifold}{\mathcal{M}}
\newcommand{\D}{\,\mathrm{d}}
\newcommand{\paren}[1]{\left( #1\right)}
\newcommand{\braces}[1]{\left\{ #1\right\}}
\newcommand{\bracket}[1]{\left[ #1\right]}

\newcommand{\alphamax}{\ensuremath{\overline{\alpha}}}
\newcommand{\deltaset}{\Delta}

\title[maxisets for spectral regularization]{Characterizations of 
variational source conditions, converse results, and maxisets of spectral regularization methods}
\author{Thorsten Hohage and Frederic Weidling}
\address{Institute for Numerical and Applied Mathematics, University of Goettingen, Lotzestr. 16-18, 37083 Goettingen, Germany}
\email{t.hohage@math.uni-goettingen.de,f.weidling@math.uni-goettingen.de}

\begin{document}

\begin{abstract}
We describe a general strategy for the verification of variational source condition 
by formulating two sufficient criteria describing the smoothness of the solution and 
the degree of ill-posedness of the forward operator in terms of a family of subspaces. 
For linear deterministic inverse problems we show that variational source conditions 
are necessary and sufficient for convergence rates of spectral regularization 
methods, which are slower than the square 
root of the noise level. A similar result is shown for linear inverse problems 
with white noise. In many cases variational source conditions can be 
characterized by Besov spaces. 
This is discussed for a number of prominent inverse problems. 
\end{abstract}

\maketitle

\section{Introduction}\label{sec:intro}
This paper is concerned with inverse problems described by 
ill-posed operator equations in real Hilbert spaces $\Xspace$ and $\Yspace$. 
Let $\nlOp:\Xspace\to \Yspace$ an injective, bounded, linear operator   
and $\xdag\in \Xspace$ the unknown solution to the inverse problem. 
We will study both a deterministic and a white noise noise model. 
In the first case measurement errors are described by a vector $\xi\in\Yspace$, and 
observed data are given by 
\begin{equation}\label{eq:deterministic_noise}
\gobs = T\xdag+\xi,\qquad \|\xi\|\leq \delta
\end{equation}
for some deterministic noise level $\delta>0$. In the second case 
measurement errors are described by a white noise process $W$ 
on $\Yspace$, and observed data are given by 
\begin{equation}\label{eq:white_noise_model}
\gobs = T\xdag+\varepsilon W 
\end{equation}
with a stochastic noise level $\varepsilon>0$. Recall that a white noise 
process is characterized by the relations $\TEW{\langle W,y\rangle}=0$ and 
$\EW{\langle W,y_1\rangle\langle W,y_2\rangle}= \langle y_1,y_2\rangle$ 
for all $y,y_1,y_2\in\Yspace$. 

Regularization theory is concerned with error estimates for approximate 
reconstruction methods (regularization methods) for $\xdag$ given data 
$\gobs$ described by 
\eqref{eq:deterministic_noise} or \eqref{eq:white_noise_model}. 
It is well-known that for ill-posed problems uniform error bounds 
necessarily require further assumptions on the solution $\xdag$ 
(see \cite[Prop.~3.11]{EHN:96}). 
Such conditions are usually called \emph{source conditions}. 
Over the last years, starting with \cite{HKPS:07} it has become increasingly 
popular to formulate such conditions in the form of variational inequalities 
\begin{equation}\label{eq:VSC_innerprod}
2\lsp \xdag,\xdag-\sol\rsp_{\Xspace} \leq \frac{1}{2}\norm{\sol-\xdag}{\Xspace}^2
+\psi\paren{\norm{\nlOp(\sol)-\nlOp(\sol^\dagger)}{\Yspace}^2}
\qquad \mbox{for all } \sol\in\Xspace 
\end{equation}
with an index function $\psi$. (A function $\psi:[0,\infty)\to [0,\infty)$ is 
called an \emph{index function} if it is continuous, strictly monotonically 
increasing, and $\psi(0)=0$.)  
Advantages of these \emph{variational source conditions (VSC)} 
over classical source conditions of the form 
$\xdag = \varphi(T^*T)w$ with an index function $\varphi$ and $w\in\Xspace$ include  
extensions to Banach spaces, general penalty and data fidelity functionals, 
treatment of nonlinear operators without the need of a derivative of $\nlOp$ 
and restrictive assumptions relating $\nlOp'$ and $\nlOp$, as well as 
simpler proofs. As a disadvantage we mention that \eqref{eq:VSC_innerprod} cannot 
be used to describe high order rates of convergence since it is easy to see 
that it cannot hold true for $\xdag\neq 0$ if $\lim_{x\to 0}\psi(x^2)/x= 0$. 
This excludes in particular the case $\psi(t)=t^{\nu}$ with $\nu>1/2$. 

In this paper we will address the following two related main questions:
\begin{itemize}
	\item What are verifiable sufficient (and possibly even necessary) conditions such that the VSC \eqref{eq:VSC_innerprod} holds true?
	\item What are necessary and sufficient conditions on $\xdag$ for a given 
	rate of convergence of a given regularization method as the noise level 
	$\delta$ or $\varepsilon$ tends to $0$? 
\end{itemize}

Let us now discuss known results from the literature and the contributions 
of this paper for both of these questions. Concerning the first question,  verifiable sufficient conditions for \eqref{eq:VSC_innerprod} have mainly been given via spectral 
source conditions so far, see Appendix \ref{sec:spectral_sc} for more details. 
In \cites{HW:15,Weidling2015} we have recently derived sufficient conditions for 
\eqref{eq:VSC_innerprod} in the form of bounds on Sobolev norms of $\xdag$ 
for nonlinear inverse medium scattering problems. 
Here we formulate in Theorem \ref{theo:general_strategy} two criteria 
which capture the main strategy of the proofs in \cites{HW:15,Weidling2015}.  
In the following we will apply them to linear inverse problems. These criteria 
describe the two main factors influencing rates of convergence: Smoothness of 
the solution and ill-posedness of the inverse of the forward operator. 
Here both criteria are formulated in terms of a sequence of approximating 
subspaces in $\Xspace$. If these spaces are chosen as eigenspaces of $T^*T$, 
we obtain an equivalent characterization of the VSC \eqref{eq:VSC_innerprod} 
in terms of the rate of decay of the spectral distribution function of $\xdag$.
The latter criterion has been introduced by Neubauer \cite{Neubauer1997} and 
shown to be necessary and sufficient for H\"older rates of convergence. 
A characterization of the VSC \eqref{eq:VSC_innerprod} for $\psi(t)=t^{\nu}$, $\nu\in (0,1/2]$ has previously been shown 
by Andreev et al.\ \cite{AEdHQS:15} using a different technique 
which seems to be limited to this special case. 



Let us now discuss the second question. In deterministic regularization theory theorems on 
the necessity of conditions for a given rate of convergence (which have already shown 
to be sufficient) are known as \emph{converse results}. 
In statistics the maximal set of $\xdag$ for which a given estimator 
achieves a given desired rate of convergence is called a \emph{maxiset}. 
Converse results for H\"older rates of Tikhonov regularization have been established by 
Neubauer \cite{Neubauer1997}. Andreev \cite{Andreev2015a} has proven converse 
results for H\"older rates of (generalized) Tikhonov regularization and 
Landweber iteration in terms of K-interpolation spaces with fine index $q=\infty$
between $\Xspace$ and $(T^*T)^{\nu}(\Xspace)$ equipped with the image space norm. 
Flemming, Hofmann \& Math{\'e} \cite{FHM:11} have derived converse results 
for general convergence rates of the bias of general spectral regularization methods 
in terms of approximate source conditions. 
Albani et al.\ \cite{Albani2015} proved converse results for general deterministic 
rates and spectral regularization method, but additional assumptions had to be 
imposed, which are not always obvious to interpret. 
Here we will prove converse results without such assumptions. As a byproduct 
of our analysis we show the equivalence of weak and strong quasioptimality 
of a posteriori parameter choice rules in many cases (see \cite{RH:07}). 
Together with our answer to the first question we also obtain converse results 
in terms of VSCs \eqref{eq:VSC_innerprod} for concave $\psi$. 
Moreover, we will show for inverse problems for which the forward operator 
satisfies $T^*T=\Lambda(-\Delta)$ for some Laplace-Beltrami operator $\Delta$,  
that VSCs for certain index functions $\psi$ are satisfied if and only if $\xdag$ belongs 
to a Besov space $B^{s}_{2,\infty}$. This holds true in particular for the backward and 
sideways heat equation, and the inverse gradiometry problem. 

In statistics maxisets of wavelet methods for the estimation of the density 
of i.i.d.\ random variables have been characterized as Besov spaces by Kerkyacharian \& Picard 
\cite{KP:93}. They consider not only $L^2$, but also other $L^p$ norms as loss functions. 
Maxisets of thresholding and more general wavelet estimators have been characterized by the 
same authors in 
\cites{KP:00,KP:02}, and their results have been generalized by Rivoirand \cite{rivoirard:04} 
to some linear estimators in the sequence space model of inverse problems. 
These latter references show in particular that under certain circumstances nonlinear 
thresholding methods have larger maxisets than linear methods for given polynomial rates. 
Here we show under fairly general assumptions that VSCs characterize maxisets of 
spectral regularization methods for linear inverse problems with white noise. 

The plan of this paper is as follows: In the following section we formulate 
and prove the theorem describing our general strategy for the verification 
of VSCs. In sections \ref{sec:converse_bias}--\ref{sec:converse_white} 
we derive converse results for the bias, rates of convergence with 
deterministic noise, and rates of convergence with white noise, respectively. 
In Section \ref{sec:spaceEquivalence} we introduce a class of inverse problems 
for which maxisets of linear spectral regularization methods are given 
by Besov spaces $B^s_{2,\infty}$. Finally, in Section \ref{sec:examples} 
we apply our theoretical results to a number of well-known inverse problem 
before we end this paper 
with some conclusions. In an appendix we show how the general strategy 
in Section \ref{sec:strategy} can be applied to verify a VSC given a spectral 
source condition for linear problems.

\section{A general strategy for verifying variational source conditions}
\label{sec:strategy}

In this section we formulate sufficient conditions for VSCs in terms of arbitrary 
families of subspaces. In the rest of this paper these will always be chosen 
as invariant subspaces of $T^*T$, but in principle the choice is arbitrary. 
To allow also polynomial, trigonometric, wavelet and other subspaces, which 
may be relevant in more general situation (see e.g.\ \cite{HW:15}), we will 
parametrize the spaces by a general index set $\projindexset$. 
\begin{theorem}\label{theo:general_strategy}
Suppose there exists a family of orthogonal projections $P_{\projparam}\in L(\Xspace)$ indexed 
by a parameter $\projparam$ in some index set $\projindexset$ such that for some 
functions $\kappa,\sigma:\projindexset\to (0,\infty)$ and some $C\geq 0$ the following 
conditions hold true for all $\projparam\in\projindexset$:
	\begin{align} \label{eq:approximation}
		&\|\xdag-P_{\projparam}\xdag\|_{\Xspace}\leq \kappa(\projparam) \\
		&\begin{aligned}\label{eq:stability}
			\lsp \xdag,P_{\projparam}(\xdag-\sol)\rsp_{\Xspace}
			\leq  \sigma(\projparam)\|\nlOp(\xdag)-\nlOp(\sol)\| +C \kappa(\projparam)\|\xdag-\sol\|&\\
			\mbox{for all }\sol\in D(\nlOp) \mbox{ with }\|\sol-\xdag\|\leq 4\|\xdag\|&.
		\end{aligned}
	\end{align}
Then $\xdag$ satisfies the VSC \eqref{eq:VSC_innerprod} with
\[
\psi(t):= 2\inf_{\projparam\in\projindexset}\left[(C+1)^2\kappa(\projparam)^2
+ \sigma(\projparam)\sqrt{t}\right].
\]
\end{theorem}

Assumption \eqref{eq:approximation} may be seen as an estimate of the approximation quality 
of the approximating spaces $P_{\projparam}\Xspace$, and assumption \eqref{eq:stability} 
may be seen as a kind of stability estimate for $\nlOp$ on these spaces. If $\xdag$ belongs to 
some smooth subspace of $\Xspace$, the stability estimate may be taken with respect to 
the norm of the dual space. 
However, \eqref{eq:stability} is not exactly a stability estimate for the restriction 
of $\nlOp$ to $P_{\projparam}\Xspace$ since by do not bound by 
$\|\nlOp(P_{\projparam}\sol)-\nlOp(P_{\projparam}\xdag)\|$, but $\|\nlOp(\sol)-\nlOp(\xdag)\|$. 
On the other hand the additional term $C\kappa(\projparam)\|\sol-\xdag\|$ may help. 
The case $C>0$ and the restriction to $\sol\in\Xspace$ 
with $\|\sol-\xdag\|\leq 4\|\xdag\|$ are also crucial for nonlinear inverse 
problems. 

\begin{proof}[Proof of Theorem \ref{theo:general_strategy}.]
If $\|\sol-\xdag\|> 4\|\xdag\|$ we have 
\begin{align}\label{eq:VSC_arg_great}
2	\lsp \xdag,\xdag-\sol\rsp 
\leq 2\|\xdag\|\;\|\xdag-\sol\| \leq \frac{1}{2} \|\xdag-\sol\|^2,
\end{align}
so \eqref{eq:VSC_innerprod} holds true. Otherwise we can apply \eqref{eq:stability} and \eqref{eq:approximation}
and the basic inequality $2ab\leq 2a^2+\frac{1}{2}b^2$
to obtain 
\begin{align*}
	2\lsp \xdag,(\xdag-\sol)\rsp 
&= 2\lsp \xdag, P_{\projparam}(\xdag-\sol)\rsp 
+ 2\lsp (I-P_{\projparam})\xdag,\xdag-\sol\rsp\\
&\leq 2\sigma(\projparam) \|\nlOp(\sol)-\nlOp(\xdag)\| + 2(C+1) \kappa(\projparam)\|\xdag-\sol\| \\
&\leq 
2\sigma(\projparam) \|\nlOp(\sol)-\nlOp(\xdag)\| + 2(C+1)^2 \kappa(\projparam)^2
+ \frac{1}{2}\|\xdag-\sol\|^2
\end{align*}
for all $\projparam\in\projindexset$. Taking the infimum over of the right hand side over 
$\projparam\in\projindexset$ with $t = \|\nlOp(\sol)-\nlOp(\xdag)\|^2$ yields 
\eqref{eq:VSC_innerprod}. 
\end{proof}

\section{Converse results for the bias}\label{sec:converse_bias}
For $\lambda\geq 0$ we define the spectral projections 
\begin{equation}
	E^{T^*T}_{\lambda}  := \chi_{[0,\lambda]}(T^*T).
\end{equation}
The function $\lambda\mapsto \|E^{T^*T}_{\lambda}f\|$ is called the 
\emph{spectral distribution function} of $f\in\Xspace$. 
For an index function $\kappa$ we define a subspace 
$\XTkappa\subset \Xspace$ via a weighted supremum norm of the spectral distribution 
function with weight $1/\kappa$:
\begin{equation}\label{eq:decaySpace}
	\begin{aligned}
		\XTkappa:=& \cbra{f \in\Xspace \colon \norm{f}{\XTkappa} < \infty},
		\qquad \qquad 
		\norm{f}{\XTkappa}:= \sup_{\lambda>0} \frac{1}{\kappa(\lambda)} \norm{E^{T^*T}_{\lambda} f}{\Xspace}
	\end{aligned}
\end{equation}
It is not difficult to see that $\XTkappa$ is a Banach space. Note that 
the unit ball in 
$\XTkappa$ contains all functions satisfying \eqref{eq:approximation} with 
$P_{\projparam}=I-E^{T^*T}_{\projparam}$.
For the remainder of this and the following two sections 
we will suppress the superscript $T^*T$ to simplify the notation. 
However, we will need it in Section \ref{sec:spaceEquivalence} to deal with 
spectral distribution functions w.r.t.\ several operators simultaneously.

\begin{theorem}\label{thm:equiv_vsc_decay}
Let $\kappa:[0,\infty)\to[0,\infty)$ be an index function 
such that $t\mapsto \kappa(t)^{2}/t^{1-\mu}$ is decreasing for some 
$\mu\in (0,1)$ and $\kappa\cdot\kappa$ is concave. Moreover, we associate 
with each such $\kappa$ an index function $\psi_{\kappa}$ by 
\begin{equation}\label{eq:kappa_psi}
\psi_{\kappa}(t):= \kappa\paren{\Theta_{\kappa}^{-1}(\sqrt{t})}^2, 
\qquad \Theta_{\kappa}(\lambda):=\sqrt{\lambda}\kappa(\lambda).
\end{equation}
Then the following statements for $\xdag\in\Xspace$ are equivalent: 
\begin{enumerate}
	\item\label{item:VSC}
	$\xdag$ satisfies a VSC with $\psi(t)=A\psi_{\kappa}(t)$	for some $A>0$. 
	\item\label{item:boundedSpectral} $\snorm{\xdag}{\Xspace^T_\kappa}<\infty$.
\end{enumerate}
More precisely, (\ref{item:VSC}) implies 
$\|E_\lambda\xdag\|\leq \sqrt{\frac{2A}{3}}\kappa\paren{\frac{2A}{3}\lambda} 
\leq  \sqrt{\frac{2A}{3}}\max(1, \sqrt{\frac{2A}{3}})\kappa(\lambda)$.  
Vice versa, if $\kappa$ is normalized such that 
$\snorm{\xdag}{\Xspace^T_\kappa}=1$, then (\ref{item:VSC}) holds true with 
$A=2(1+\mu^{-1})+ 2\kappa(\|T\|^2)\sup_{t\in(0,4\Tnorm{T}{}\Tnorm{\xdag}{}]} \sqrt{t}/\psi_{\kappa}(t)$, 
and $A$ is finite. 
\end{theorem}

\begin{proof}

\emph{(\ref{item:VSC}) $\Rightarrow$ (\ref{item:boundedSpectral}):} 
First note that 
\begin{equation}\label{eq:psi_alt}
\psi_{\kappa}^{-1}(\xi) = 
\xi\cdot(\kappa\cdot\kappa)^{-1}(\xi)
\end{equation}
since 
\[
\psi_{\kappa}(t)\cdot (\kappa \cdot\kappa)^{-1}(\psi_\kappa(t))
=\kappa(\Theta_{\kappa}^{-1}(\sqrt{t}))^2 \cdot \Theta_{\kappa}^{-1}(\sqrt{t})
= \Theta_{\kappa}(\Theta_{\kappa}^{-1}(\sqrt{t}))^2 =\sqrt{t}^2=t. 
\]
Choosing $\sol = (I-E_{\lambda})\xdag$ in \eqref{eq:VSC_innerprod} yields
\[
2\|E_{\lambda}\xdag\|^2 = 2\lsp \xdag,E_{\lambda}\xdag\rsp 
\leq \frac{1}{2}\|E_{\lambda}\xdag\|^2 + A\psi_{\kappa}\paren{\|TE_{\lambda}\xdag\|^2}.
\] 
As 
\[
\|TE_{\lambda}\xdag\|^2 = \int_0^\lambda \lambda \D\|E_{\lambda}\xdag\|^2
\leq \lambda  \int_0^\lambda \D\|E_{\tilde{\lambda}}\xdag\|^2
= \lambda \|E_{\tilde{\lambda}}\xdag\|^2,
\]
the spectral distribution function $\tilde{\kappa}(\lambda):=\|E_{\lambda}\xdag\|$
of $\xdag$ satisfies the inequality
\(
\frac{3}{2}\tilde{\kappa}(\lambda)^2\leq A\psi_{\kappa}\paren{\lambda\tilde{\kappa}(\lambda)^2}.
\)
Hence, setting $\widetilde{\psi}_{\kappa}(t):=\psi_{\kappa}(t)/t$ we have  
\[
\frac{3}{2A\lambda}\leq \frac{\psi_{\kappa}\paren{\lambda\tilde{\kappa}(\lambda)^2}}{\lambda\tilde{\kappa}(\lambda)^2}=\widetilde{\psi}_{\kappa}(\lambda\tilde{\kappa}(\lambda)^2).
\]
As $\kappa\cdot\kappa$ is assumed to be a concave index function, 
the inverse function $(\kappa\cdot\kappa)^{-1}$ is a convex index function. 
This implies that $\xi\mapsto\xi\cdot(\kappa\cdot\kappa)^{-1}(\xi)$ is a 
convex index function as well, and from \eqref{eq:psi_alt} we see 
that $\psi_{\kappa}$ is concave. This in turn implies that $\widetilde{\psi}_{\kappa}$ 
is monotonically decreasing, so we obtain
\[
\widetilde{\psi}_{\kappa}^{-1}\paren{\frac{3}{2A\lambda}}
\geq \lambda\tilde{\kappa}(\lambda)^2
\]
or with $\beta= \frac{2A\lambda}{3}$
\begin{equation}\label{eq:aux_VSC_decay}
\tilde{\kappa}\paren{\frac{3}{2A}\beta}\leq \sqrt{\frac{2A}{3}}
\sqrt{\frac{1}{\beta}\widetilde{\psi}_{\kappa}^{-1}\paren{\frac{1}{\beta}}} 
\end{equation}
for all $\beta>0$. Setting $\xi=\kappa(\beta)^2$ in \eqref{eq:psi_alt}, 
we obtain $\psi_{\kappa}^{-1}(\kappa(\beta)^2) = \beta\kappa(\beta)^2$ or 
$\kappa(\beta)^2 = \psi_\kappa(\beta\kappa(\beta)^2)$. 
This is equivalent to $\frac{1}{\beta}=\frac{\psi_{\kappa}(\beta\kappa(\beta)^2)}
{\beta\kappa(\beta)^2}$ and to $\frac{1}{\beta}\widetilde{\psi}_{\kappa}^{-1}
\paren{\frac{1}{\beta}}= \kappa^2(\beta)$. 
Plugging this into \eqref{eq:aux_VSC_decay} shows that 
$\tilde{\kappa}(\lambda)\leq \sqrt{\frac{2A}{3}}\kappa\paren{\frac{2A}{3}\lambda}$ for all $\lambda>0$. Due to the concavity of the index function $\kappa\cdot \kappa$ we have
$\kappa\paren{\frac{2A}{3}\lambda}^2\leq \kappa(\max(1,\frac{2A}{3})\lambda)^2
\leq \max(1,\frac{2A}{3})\kappa(\lambda)^2$.  

\emph{(\ref{item:boundedSpectral}) $\Rightarrow$ (\ref{item:VSC}):} Suppose that $\snorm{\xdag}{\Xspace^T_\kappa}=1$, i.e.\ $\|E_\lambda\xdag\|\leq \kappa(\lambda)$ for all 
$\lambda>0$. In a first step we show that 
\begin{equation}\label{eq:Tinvf_bound}
\|(T(I-E_{\lambda}))^{\dagger}\xdag\|^2 
= \int_{\lambda}^{\|T^*T\|}\frac{1}{t}\D\|E_t\xdag\|^2 
\leq \frac{1}{\mu}\frac{\kappa(\lambda)^2}{\lambda} 
+ \|\xdag\|^2
\end{equation}
for all $\lambda>0$. Here $(T(I-E_{\lambda}))^{\dagger}$ denotes 
the Moore-Penrose inverse of $T(I-E_{\lambda})$. 
By partial integration we have 
\[
 \int_{\lambda}^{\|T^*T\|}\frac{1}{t}\D\|E_t\xdag\|^2 
= \|\xdag\|^2 - \|E_{\lambda}\xdag\|^2
+ \int_{\lambda}^{\|T^*T\|}\frac{\|E_t\xdag\|^2}{t^2}\D t. 
\]
Now we use the assumption $\|E_t\xdag\|\leq \kappa(t)$ and 
the monotonicity of $\kappa(t)^{2}/t^{1-\mu}$ to obtain 
\begin{align*}
\int_{\lambda}^{\|T^*T\|}\frac{\|E_t\xdag\|^2}{t^2}\D t 
&\leq \int_{\lambda}^{\|T^*T\|}\frac{\kappa(t)^{2}}{t^{1-\mu}}
\frac{1}{t^{1+\mu}}\D t 
\leq \frac{\kappa(\lambda)^2}{\lambda^{1-\mu}}
\int_{\lambda}^{\|T^*T\|}\frac{1}{t^{1+\mu}}\D t\\
&= \frac{\kappa(\lambda)^2}{\mu\lambda}\paren{1-\frac{\lambda^\mu}{\|T^*T\|^\mu}} 
\leq \frac{\kappa(\lambda)^2}{\mu\lambda}.
\end{align*}

In a second step we can now use \eqref{eq:Tinvf_bound} 
to verify assumption \eqref{eq:stability} in Theorem \ref{theo:general_strategy}:
\begin{align*}
&\lsp \xdag,(I-E_{\lambda})(\xdag-\sol)\rsp 
= \lsp (T(I-E_{\lambda}))^{\dagger}\xdag,T(I-E_{\lambda})(\xdag-\sol)\rsp\\
&\qquad\qquad \leq \|T(I-E_{\lambda}))^{\dagger}\xdag\|\,\|T(I-E_{\lambda})(\xdag-\sol)\|
\leq \paren{\frac{\kappa(\lambda)}{\sqrt{\mu\lambda}}+\|\xdag\|} \|T\xdag-T\sol\|,
\end{align*}
i.e.\ $\sigma(\lambda)=\frac{\kappa(\lambda)}{\mu\lambda}+\|\xdag\|$. 
Hence, by Theorem \ref{theo:general_strategy} \eqref{eq:VSC_innerprod} holds true with 
\begin{align*}
\psi(t) &= 2\inf_{\lambda>0}\left[\kappa(\lambda)^2+
\paren{\frac{\kappa(\lambda)}{\sqrt{\mu\lambda}}+\|\xdag\|}\sqrt{t}\right] 
\leq 2\paren{1+\frac{1}{\sqrt{\mu}}}\psi_{\kappa}(t) 
+ 2\|\xdag\|\sqrt{t}
\end{align*}
where we have chosen $\lambda = \Theta_{\kappa}^{-1}(\sqrt{t})$, i.e.\ 
$\sqrt{\lambda}\kappa(\lambda) =\sqrt{t}$. This implies 
$\frac{\kappa(\lambda)}{\sqrt{\lambda}}\sqrt{t}=\kappa^2(\lambda)
=\psi_{\kappa}(t)$. 
It remains to bound $\sqrt{t}$ in terms of $\psi_{\kappa}(t)$. 
Note from \eqref{eq:VSC_arg_great} that we only need to show the variational 
source condition for $\Tnorm{\xdag-\sol}{}\leq 4\Tnorm{\xdag}{}$ in order to prove it everywhere. Hence it is enough to bound $\sqrt t$ by $\psi_\kappa(t)$ for 
$\sqrt{t}=\Tnorm{T\xdag-T\sol}{}\leq 4\Tnorm{T}{} \Tnorm{\xdag}{}$.
We have
\begin{equation*}
	\|\xdag\|\sqrt{t} \leq \kappa(\|T\|^2) \psi_{\kappa}(t)
	\sup_{\tau\in(0,4\Tnorm{T}{}\Tnorm{\xdag}{}]} \frac{\sqrt{\tau}}{\psi_{\kappa}(\tau)}.
\end{equation*}
To see that this
is finite and even $\lim_{\tau\to 0}\sqrt{\tau}/\psi_{\kappa}(\tau)=0$, 
we substitute $\delta=\Theta_{\kappa}^{-1}(\sqrt{\tau})$:
\[
\lim_{\tau\to 0}\frac{\sqrt{\tau}}{\psi_{\kappa}(\tau)}
=\lim_{\tau\to 0}\frac{\sqrt{\tau}}{\kappa(\Theta_{\kappa}^{-1}(\sqrt{\tau}))^2}
= \lim_{\delta\to 0}\frac{\Theta_{\kappa}(\delta)}{\kappa(\delta)^2}
\leq \lim_{\delta\to 0}\delta^{\mu/2}\sqrt{\frac{\delta^{1-\mu}}{\kappa^2(\delta)}}=0. 
\qedhere
\]
\end{proof}

We now consider spectral regularization methods of the form 
\begin{equation}\label{eq:regmethod}
\widehat{\sol}_{\alpha}:= R_{\alpha}\gobs\qquad \mbox{with}\qquad 
R_{\alpha}:=q_{\alpha}(T^*T)T^*
\end{equation}
and impose the following assumptions:
\begin{assumption}\label{ass:SR}
With $r_{\alpha}(\lambda):=1-\lambda q_{\alpha}(\lambda)$ assume that there are 
constants $\Cq>0$, $\alphamax\in(0,\infty]$, and 
$0<\Cdiaglow\leq \Cdiag<1$  such that 
\begin{enumerate} 
	\item\label{it:reg_q} $|q_{\alpha}(\lambda)|\leq \frac{\Cq}{\alpha}$ 
	for all $\lambda\in[0,\Tnorm{T^*T}{}]$ ,
	\item\label{it:reg_decr}
	$\lambda\mapsto r_{\alpha}(\lambda)$ is decreasing for all $\alpha>0$
	and $r_{\alpha}(\lambda)\geq 0$,
	\item\label{it:reg_incr} 
	$\alpha\mapsto r_{\alpha}(\lambda)$ is increasing 
	for all $\lambda\in[0,\Tnorm{T^*T}{}]$, 
	\item\label{it:reg_diag} $\Cdiaglow\leq r_{\alpha}(\alpha)\leq \Cdiag$ for all 
$0<\alpha\leq \alphamax$. 
\end{enumerate}
\end{assumption}
As $r_\alpha(0) =1-0 q_\alpha(0)=1$, assumption (\ref{it:reg_decr}) implies 
\begin{equation}\label{eq:r_bound}
0\leq r_{\alpha}(\lambda)\leq 1
\end{equation}
for all $\alpha>0$ and $\lambda\geq 0$. 
Below we will use the following notations for $x,y\in\Rset$:
\[
\lmax{x}{y}:=\max(x, y),\qquad \lmin{x}{y}:=\min(x,y)
\]
Assumption \ref{ass:SR} is satisfied in particular for the following methods. 
Unless stated otherwise we choose $\alphamax=\infty$.
For a more detailed discussion of these methods we refer to \cite{EHN:96}. 
\begin{itemize}
\item \emph{Tikhonov regularization:} 
Here $q_{\alpha}(\lambda)= (\alpha+\lambda)^{-1}$ and 
$r_{\alpha}(\lambda) = \alpha/(\alpha+\lambda)$. We have $\Cq=1$ and 
$\Cdiaglow=\Cdiag=\frac{1}{2}$. 
	\item \emph{Showalter's method:} Here $r_{\alpha}(\lambda)=\exp(-\lambda/\alpha)$, 
	$\Cq=1$ and $\Cdiaglow=\Cdiag=\exp(-1)$. 
	\item \emph{Landweber iteration:} For $\alpha>0$ let 
	$k_{\alpha}:=\min\{n\in\Nset_0:n+1> 1/\alpha\}$ be the number of iterations. 
	Then $r_{\alpha}(\lambda)= (1-\mu\lambda)^{k_{\alpha}}$ and
  $q_{\alpha}(\lambda)= \sum_{j=0}^{k_{\alpha}-1}(1-\mu\lambda)^j$ 
	where $0<\mu\leq \|T^*T\|^{-1}$ is the step length parameter. 
	We have $\Cq=1$.
	For $\alpha=1/(n+\epsilon)$ with $\epsilon\in[0,1)$ we have $k_\alpha=n$, therefore $r_\alpha(\alpha)\geq (1-\mu/n)^{n}$ which by the inequality of arithmetic and geometric means is monotonically increaing in $n=k_\alpha$ for $k_\alpha> \mu$, hence we choose $\alphamax<\lmin{\Tnorm{T^*T}{}}{1}$ and get $\Cdiaglow=(1-\mu/ k_{\alphamax})^{k_{\alphamax}}$ and $\Cdiag=\lim_{n\to\infty}(1-\mu/(n+1))^{n} 
	=\exp(-\mu)$.
	\item \emph{$k$-times iterated Tikhonov regularization:} 
	This is described by $r_{\alpha}(\lambda)= \alpha^k/(\alpha+\lambda)^k$. 
	We have $\Cq=k$ and $\Cdiaglow=\Cdiag=2^{-k}$. 
	\item \emph{Lardy's method:} Here $r_{\alpha}(\lambda)=\beta^{k_{\alpha}}/
	(\beta+\lambda)^{k_{\alpha}}$ where $\beta>0$ is fixed and the iteration number 
	$k_{\alpha}$ and $\Cq=1$. Choosing $\alphamax:=\lmin{1}{\beta}$ 
	we have $\Cdiag=\exp(-1/2\beta)$). Choosing $\alpha$ as for Landweber we see that $r_\alpha(\alpha)\geq (1+1/(\beta n))^{-n}$, therefore  $\Cdiaglow=\exp(-1/\beta)$ since $(1+1/(\beta n))^{n}\rightarrow\exp(1/\beta)$ is monotonically increasing in $n$ as argued above.
	\item \emph{modified spectral cutoff:} Here $r_{\alpha}(\lambda)=(1-\lambda/2\alpha)\lor 0$, $q_{\alpha}(\lambda) = 1/\lambda\land 1/2\alpha$, and 
	$\Cq=\Cdiaglow=\Cdiag=1/2$. 
\end{itemize}
Note that Assumption \ref{ass:SR}(\ref{it:reg_diag}) is violated 
for standard spectral cutoff (or truncated SVD), i.e.\ 
$r_{\alpha}(\lambda)=1$ if $\alpha\leq \lambda$ and $r_{\alpha}(\lambda)=0$ else.

\begin{theorem}[{\cite[Prop.~2.3]{Albani2015}}]\label{thm:biasDecayEquivalence}
Assume that a spectral regularization method satisfies Assumption \ref{ass:SR}.  
Moreover, assume that for the index function $\kappa$ there exists $A>0$ 
and $\nu>1$ such that 
\begin{equation}\label{eq:qualif_assump}
r_{\alpha}(\lambda)\kappa(\lambda)^\nu\leq B \kappa(\alpha)^\nu
\end{equation}
for all $\alpha,\lambda>0$ (i.e.\ the qualification of the regularization 
method covers $\kappa^{\nu}$ in the terminology of \cite{MP:03a}). 
Then the following statements for $\xdag\in\Xspace$ are equivalent:
\begin{enumerate}
	\item\label{item:albaniSpace} $\|\xdag\|_{\Xspace^T_\kappa}<\infty$. 
	\item\label{item:albaniBias} $A:=\sup_{0<\alpha\leq\alphamax}\frac{1}{\kappa(\alpha)}\|r_{\alpha}(T^*T)\xdag\|<\infty$, 
	i.e.\  the bias for $\xdag$ is of order $\mathcal{O}(\kappa(\alpha))$.
\end{enumerate}
More precisely, 
\begin{align*}
&\|\xdag\|_{\Xspace^T_\kappa}\leq 
\lmax{\frac{A}{\Cdiaglow}}{\frac{\|\xdag\|}{\kappa(\alphamax)}}
\quad \mbox{and}\quad
A^2\leq \frac{B\|\xdag\|^2}{\kappa(\|T\|^2)}+\|\xdag\|_{\Xspace^T_\kappa}^2
\paren{1+\frac{B^{1/\nu}\nu \Cdiag^{(\nu-1)/\nu}}{\nu-1}}.
\end{align*}
\end{theorem}
\begin{proof}
The more difficult implication (\ref{item:albaniSpace}) $\Rightarrow$ (\ref{item:albaniBias}) has been proved in 
\cite[Prop.~2.3]{Albani2015}. 
Since we have slightly different assumptions we give the proof 
of the implication (\ref{item:albaniBias}) $\Rightarrow$ (\ref{item:albaniSpace}). For $0<\alpha\leq\alphamax$ 
we have 
\[
\|E_{\alpha}\xdag\|
\leq \frac{1}{r_{\alpha}(\alpha)}\|r_{\alpha}(T^*T)\xdag\|
\leq \frac{A}{\Cdiaglow}\kappa(\alpha).
\]
Otherwise, if $\alpha>\alphamax$, we have
\(
\|E_{\alpha}\xdag\|/\kappa(\alpha) 
\leq \|\xdag\|/\kappa(\alphamax). 
\)
\end{proof}

Recall that the largest number $\mu_0>0$ for which \eqref{eq:qualif_assump} 
holds true for $\kappa(\alpha)^\nu=\alpha^{\mu_0}$ is called the 
\emph{classical qualification} of the regularization method. We have 
$\mu_0=k$ for $k$-times iterated Tikhonov regularization and $\mu=\infty$ 
for Showalter's method, Lardy's methods, Landweber iteration, 
and modified spectral cutoff (\cite{EHN:96}). 


\section{Converse results for deterministic noise}\label{sec:converse_det}
This section discusses regularization methods for the deterministic 
noise model \eqref{eq:deterministic_noise}. 
\begin{theorem}\label{thm:det_noise_equiv}
Assume that a spectral regularization method satisfies Assumption \ref{ass:SR}. 
Moreover, let $\kappa$ be an index function for which there exists $p\geq 1$ 
such that 
\begin{equation}\label{eq:kappa_growth}
\kappa(r\alpha)\leq r^p\kappa(\alpha)
\end{equation} 
for all $\alpha>0$ and $r\geq 1$ 
(i.e.\ $\kappa$ does not grow faster than polynomially). 
Then the following statements are equivalent for $\xdag\in\Xspace$: 
\begin{enumerate}
	\item\label{item:dneSpace} 
$A:=\sup_{0<\alpha\leq\alphamax} \frac{1}{\kappa(\alpha)^2}\|r_{\alpha}(T^*T)\xdag\|^2<\infty$.
	\item\label{item:dneRate} 
$B:= \sup_{0<\delta\leq\Theta_{\kappa}(\alphamax)} \frac{1}{\psi_{\kappa}(\delta^2)}
	\inf_{0<\alpha\leq\alphamax}\sup_{\|\xi\|\leq\delta}\|R_{\alpha}(T\xdag+\xi)-\xdag\|^2 
	<\infty$.
\end{enumerate}
More precisely, 
\[
B\leq 2(A+\Cq)\qquad \mbox{and}\qquad 
A\leq \lmax{B\paren{\lmax{\frac{4B^2}{(1-\Cdiag)^4}}{1}}^{p}}
{\Tnorm{\xdag}{}^2\kappa\rbra{\frac{\alphamax (1-\Cdiag)^2}{2B}}^{-2}}.
\]
\end{theorem}

\begin{proof}
\emph{(\ref{item:dneSpace}) $\Rightarrow$ (\ref{item:dneRate}):}  
From the standard estimate 
\begin{equation}\label{eq:Ralpha_norm}
\|R_{\alpha}\|^2=\|R_{\alpha}^*R_{\alpha}\| 
=\|q_{\alpha}(T^*T)^2T^*T\|
\leq \|q_{\alpha}\|_{\infty}\|1-r_{\alpha}\|_{\infty}
\leq \frac{\Cq}{\alpha} 
\end{equation} 
using Assumption \ref{ass:SR}(\ref{it:reg_q}) and \eqref{eq:r_bound}. Hence we have 
\[
\|R_{\alpha}(T\xdag+\xi)-\xdag\|^2
\leq \paren{\|r_{\alpha}(T^*T)\xdag\|  + \|R_{\alpha}\|\delta}^2
\leq 2A\kappa(\alpha)^2+2\frac{\Cq \delta^2}{\alpha} 
\] 
for all $\|\xi\|\leq \delta$ and $0<\alpha\leq\alphamax$. 
Choosing $\alpha = \Theta_{\kappa}^{-1}(\delta)$
and using 
$\sqrt{\Theta_{\kappa}^{-1}(\delta)}\kappa(\Theta_{\kappa}^{-1}(\delta))
=\Theta_{\kappa}(\Theta_{\kappa}^{-1}(\delta))=\delta$, i.e.\ 
$\delta^2/\Theta_{\kappa}^{-1}(\delta)= \psi_{\kappa}(\delta^2)$, 
we obtain
\[
\sup_{\|\xi\|\leq \delta} \|R_{\alpha}(T\xdag+\xi)-\xdag\|^2
\leq (2A+2\Cq)\psi_{\kappa}(\delta^2).
\]

\emph{(\ref{item:dneRate}) $\Rightarrow$ (\ref{item:dneSpace}):}  Expanding 
\[
\|R_{\alpha}(T\xdag+\xi)-\xdag\|^2 = \|r_{\alpha}(T^*T)\xdag+R_{\alpha}\xi\|^2
= \|r_{\alpha}(T^*T)\xdag\|^2+ 2\lsp r_{\alpha}(T^*T)\xdag,R_{\alpha}\xi\rsp + \|R_{\alpha}\xi\|^2,
\]
we see that only the middle of the three terms on the right hand side 
is affected by a sign change of $\xi$. Therefore, to bound the 
supremum over $\xi$ from below, we may assume that the middle term is 
positive and neglect it to obtain
\begin{equation}\label{eq:lowerbound1}
\sup_{\|\xi\|\leq \delta}\|R_{\alpha}(T\xdag+\xi)-\xdag\|^2
\geq \|r_{\alpha}(T^*T)\xdag\|^2+\sup_{\|\xi\|\leq \delta}\|R_{\alpha}\xi\|^2
= \|r_{\alpha}(T^*T)\xdag\|^2+\|R_{\alpha}\|^2\delta^2.
\end{equation}
From the equality in \eqref{eq:Ralpha_norm} and the isometry of the functional calculus together with the last point in Assumption \ref{ass:SR} we obtain
\[
\|R_{\alpha}\|^2
=\sup_{\lambda\geq 0} \lambda |q_{\alpha}(\lambda)|^2
= \sup_{\lambda\geq 0} \frac{(1-r_\alpha(\lambda))^2}{\lambda}
\geq \frac{(1-r_\alpha(\alpha))^2}{\alpha}
\geq \frac{(1-\Cdiag)^2}{\alpha}
\]
if $0<\alpha\leq\alphamax$. 
By Assumption \ref{ass:SR}(\ref{it:reg_incr})  
the first term on the right hand side of \eqref{eq:lowerbound1} 
is increasing in $\alpha$ whereas the second term is decreasing. 
Therefore, using the choice 
$\alpha^*(\delta) = \Theta_{\kappa}^{-1}(\delta)(1-\Cdiag)^2/(2B) $ 
from the first part of the proof, for which both terms are of the 
same order, we obtain the lower bound 
\begin{align*}
B\psi_{\kappa}(\delta^2)
&\geq
\inf_{\alpha>0}\sup_{\|\xi\|\leq\delta}\|R_{\alpha}(T\xdag+\xi)-\xdag\|^2 
\geq \lmin{\|r_{\alpha^*(\delta)}(T^*T)\xdag\|^2}
{\|R_{\alpha^*(\delta)}\|^2\delta^2}\\
&\geq \lmin{\|r_{\alpha^*(\delta)}(T^*T)\xdag\|^2}
{\frac{(1-\Cdiag)^2}{\alpha^*(\delta)}\delta^2}
= \lmin{\|r_{\alpha^*(\delta)}(T^*T)\xdag\|^2}
{2B\psi_{\kappa}(\delta^2)}
\end{align*}
for $\alphamax\geq\alpha^*(\delta)>0$. As 
$\|r_{\alpha^*(\delta)}(T^*T)\xdag\|^2\geq 2B\psi_{\kappa}(\delta^2)$ 
would lead to a contradiction, the minimum is attained at the first argument, 
and we have 
$\|r_{\alpha^*(\delta)}(T^*T)\xdag\|^2\leq B\psi_{\kappa}(\delta^2)$. 
As $\delta=\Theta_{\kappa}(\frac{{2B}}{(1-\Cdiag)^2}\alpha^*(\delta))$ and 
$\psi_{\kappa}(t)=(\kappa\circ\Theta_{\kappa}^{-1}(\sqrt{t}))^2$, 
we obtain
\[
\|r_{\alpha^*}(T^*T)\xdag\|^2 \leq B\kappa\paren{\frac{{2B}}{(1-\Cdiag)^2}\alpha^*}^2
\leq B\paren{\lmax{\frac{4B^2}{(1-\Cdiag)^4}}{1}}^{p}\kappa(\alpha^*)^2
\]
for all $0<\alpha^* \leq \alphamax(1-\Cdiag)^2/(2B) $.
If $\alphamax(1-\Cdiag)^2/(2B) <\alpha\leq \alphamax$ we can bound
\begin{equation*}
	\kappa(\alpha)^{-2}\|r_{\alpha}(T^*T)\xdag\|^2
\leq \kappa\rbra{\frac{\alphamax (1-\Cdiag)^2}{2B}}^{-2}\Tnorm{\xdag}{}^2
\end{equation*}
finishing the proof.
\end{proof}

We point out that in comparison to similar results by 
Neubauer \cite[Thm.~2.6]{Neubauer1997} and Albani et al.\ 
\cite[Prop.~3.3]{Albani2015} we have interchanged the order 
of the supremum over the noise vector $\xi$ and the infimum over the 
regularization parameter $\alpha$. Since obviously 
\begin{equation}\label{eq:trivial_infSupSwap}
\sup_{\|\xi\|\leq \delta}\inf_{\alpha}\|R_{\alpha}(T\xdag+\xi)-\xdag\|\leq 
\inf_{\alpha}\sup_{\|\xi\|\leq\delta}\|R_{\alpha}(T\xdag+\xi)-\xdag\|,
\end{equation}
the more difficult implication $(\ref{item:dneRate}) \Rightarrow (\ref{item:dneSpace})$ in 
Theorem \ref{thm:det_noise_equiv} is weaker than in \cite{Albani2015}. 
However, we do not have to impose additional assumptions relating 
the regularization method and the index function as required in \cite{Albani2015}. 
Let us now state conditions under which a reverse inequality to 
\eqref{eq:trivial_infSupSwap} holds true:
\begin{lemma}\label{lem:infSupSwap}
	Under Assumption \ref{ass:SR} the estimate
	\begin{equation}\label{eq:infSupSwap}
		\inf_{0<\alpha\leq\alphamax} \sup_{\norm{\xi}{} \leq \delta} \norm{R_{\alpha}(T\xdag+\xi)-\xdag}{}
		\leq 2\sqrt{2} \sup_{\norm{\xi}{} \leq \delta} \inf_{0<\alpha\leq\alphamax} \norm{R_{\alpha}(T\xdag+\xi)-\xdag}{}
	\end{equation}
	holds true for all 
	\begin{equation}\label{eq:defi_deltaset}
\delta\in		\deltaset(\xdag):=\cbra{\frac{\Tnorm{r_{\alpha}(T^*T)\xdag}{}}{\Tnorm{R_{\alpha}}{}}: 0<\alpha<\alphamax}
	\end{equation}
This set has the following properties:
\begin{enumerate}
\item 
If $q_\alpha(\lambda)$ is continuous in $\alpha$ with $\alphamax=\infty$ for all $\lambda\in\sigma(T^*T)$ 
and $\xdag\neq 0$, then $\deltaset(\xdag)= (0,\infty)$. 
\item 
If $E_\alpha \xdag\neq 0$ for all $\alpha>0$, 
then $0$ is always a cluster point of $\deltaset(\xdag)$. 
\item For Landweber iteration with $\mu\|T^*T\|<1$ and Lardy's method 
and $\xdag\neq 0$ the size of the gaps of $\deltaset(\xdag)$ on a logarithmic 
scale is bounded by $\ln \gamma$ with 
\begin{equation}\label{eq:gaps_deltaset}
\gamma:=\sup\braces{\frac{b}{a}:a,b\in\overline{\deltaset(\xdag)} 
\land 0<a<b\land
(a,b)\cap \deltaset(\xdag)=\emptyset}<\infty
\end{equation}
\end{enumerate}
\end{lemma}
\begin{proof}

	For $\delta\in\deltaset(\xdag)$ there exists 
	$\alpha^\prime=\alpha^\prime(\delta,\xdag)$ such that
	\begin{equation*}
		\norm{r_{\alpha^\prime}(T^*T)\xdag}{}=\norm{R_{\alpha^\prime}}{} \delta.
	\end{equation*}
By the definition of the operator norm, for each $\epsilon>0$ there exists 
a noise vector $\xi'$ with $\Tnorm{\xi'}{}\leq \delta$ such that 
$\Tnorm{R_{\alpha^\prime}\xi^\prime}{}\geq (1-\epsilon)\Tnorm{R_{\alpha^\prime}}{}\delta$ (if $T$ is compact we may choose $\epsilon=0$). We claim that 
$\xi'$ (depending on $\alpha^{\prime}$ and $\xdag$) can be chosen such that 
	\begin{equation}\label{eq:signXiprime}
		\pairing{r_\alpha(T^*T)\xdag}{R_\alpha \xi^\prime}\geq 0\qquad 
		\mbox{for all } \alpha \in (0,\alphamax].
	\end{equation}
	Let $T= U(T^*T)^{1/2}$ be the polar decomposition of $T$ with a unitary 
	operator $U:\Xspace\to \overline{R(T)}\subset \Yspace$ 
	(recall that $T$ is assumed to be injective). As $\overline{R(T)}
	\supset N(R_{\alpha})^{\perp}$, we may assume w.l.o.g.\ that 
	$\xi^{\prime} \in \overline{R(T)}$. 
	Hence, \eqref{eq:signXiprime} is equivalent to 
	$\pairing{r_\alpha(T^*T)\xdag}{q_{\alpha}(T^*T)(T^*T)^{1/2} \xi^{\prime\prime}}\geq 0$ with $\xi^{\prime}=U\xi^{\prime\prime}$. By the Halmos version of the 
	spectral theorem \cite{halmos:63}, 
	$T^*T$ is unitarily equivalent to a multiplication operator 
	$M_{\lambda}:L^2(\Omega,\mu)\to L^2(\Omega,\mu)$, 
	$(M_{\lambda}h)(z) = \lambda(z)h(z)$, $z\in\Omega$ 
	on a locally compact space $\Omega$ with positive Borel measure $\mu$ and 
	a non-negative function  
	$\lambda\in L^{\infty}(\Omega,\mu)$, i.e.\ $T^*T = W^*M_{\lambda} W$ for some 
	unitary operator $W:\Xspace\to L^2(\Omega,\mu)$ (if $T$ is compact, $\mu$ may 
	be chosen as counting measure on $\Omega=\Nset$). 
	It follows that  
	\[
			\pairing{r_\alpha(T^*T)\xdag}{R_\alpha \xi^\prime}
			= \int_{\Omega} r_{\alpha}\paren{\lambda(z)}\; \paren{W\xdag}(z) \;
			q_{\alpha}\paren{\lambda(z)}
			\sqrt{\lambda(z)} \;\paren{W\xi^{\prime\prime}}(z)\,d\mu(z).
	\]
  By Assumption \ref{ass:SR} we have $r_{\alpha}\geq 0$ and $q_{\alpha}\geq 0$ 
	for all $\alpha>0$ (see \eqref{eq:r_bound}).  
	Therefore, the right hand side of the last equation is non-negative if 
	$(W\xdag)(z) (W\xi^{\prime\prime})(z)\geq 0$ for $\mu$-almost all $z\in\Omega$. 
  This may be achieved by replacing $(W\xi^{\prime\prime})(z)$ by 
	$s(z)(W\xi^{\prime\prime})(z)$ with a measurable function $s:\Omega\to \{-1,1\}$. 
	This shows that \eqref{eq:signXiprime} holds true if $\xi^{\prime}$ is 
	replaced by $UW^*(s\cdot (W\xi^{\prime\prime}))$. 
	
	With this choice of $\alpha^\prime$ and $\xi^{\prime}$ we can estimate
	\begin{align*}
		\inf_{\alpha>0} \sup_{\norm{\xi}{} \leq \delta} \norm{R_{\alpha}(T\xdag+\xi)-\xdag}{}
		&\leq \inf_{\alpha>0} \sup_{\norm{\xi}{} \leq \delta} 
		\left[\norm{r_{\alpha}(T^*T)\xdag}{}+\norm{R_\alpha \xi}{}\right] \\
		&\leq  \norm{r_{\alpha^\prime}(T^*T)\xdag}{}+
		 \sup_{\norm{\xi}{} \leq \delta}\norm{R_{\alpha^\prime} \xi}{}\\
   &\leq		 \norm{r_{\alpha^\prime}(T^*T)\xdag}{}+
	\frac{1}{1-\epsilon}\norm{R_{\alpha^\prime} \xi^\prime}{}\\
		&\leq \frac{2}{1-\epsilon} \inf_{0<\alpha\leq \alphamax}
		\left[\norm{r_{\alpha}(T^*T)\xdag}{}+\norm{R_\alpha \xi^\prime}{}\right]
	\end{align*}
	since the first term in the last line is monotonically increasing and the second term is monotonically decreasing in $\alpha$. 
As $(\Tnorm{x}{}+\Tnorm{y}{})^2 \leq 2 \Tnorm{x+y}{}^2$ for $\langle x, y \rangle\geq 0$, we obtain via \eqref{eq:signXiprime} that
	\begin{align*}
		\inf_{\alpha>0} \sup_{\norm{\xi}{} \leq \delta} \norm{R_{\alpha}(T\xdag+\xi)-\xdag}{}
		&\leq \frac{2\sqrt{2}}{1-\epsilon} \inf_{0<\alpha\leq \alphamax} \norm{R_{\alpha}(T\xdag+\xi^\prime)-\xdag}{}\\
		&\leq \frac{2\sqrt{2}}{1-\epsilon} \sup_{\norm{\xi}{} \leq \delta} \inf_{0<\alpha\leq \alphamax} \norm{R_{\alpha}(T\xdag+\xi)-\xdag}{}.
	\end{align*}
	As $\epsilon>0$ was arbitrary, we have proven \eqref{eq:infSupSwap}. Let us now show the properties of $\deltaset(\xdag)$: 
	\begin{enumerate}
	\item 
	If $q_{\alpha}(\lambda)$ and $r_{\alpha}(\lambda)$ are continuous in $\alpha$, 
	then $\|r_{\alpha}(T^*T)\xdag\|$ is continuous in $\alpha$ by Lebesgue's 
	Dominated Convergence Theorem, and so is $\alpha\mapsto \|R_{\alpha}\|$. 
	Moreover, $\lim_{\alpha\to 0}\|r_{\alpha}(T^*T)\xdag\|=0$ as 
	$T$ is assumed to be injective, $\alpha\mapsto \|R_{\alpha}\|$ is 
	decreasing, by \eqref{eq:Ralpha_norm} we have $\|R_{\alpha}\|\to 0$ as 
	$\alpha\to \alphamax$, and by \eqref{eq:r_bound} we have 
	$\|r_{\alpha}(T^*T)\xdag\|\leq\|\xdag\|$ for all $\alpha$. As $\xdag\neq 0$,  
	the case $r_{\alpha}(T^*T)\xdag =0$ for all $\alpha>0$ can be excluded by 
	noting that 
\begin{equation}\label{eq:rneq0}
r_{\alpha}(\lambda)>0\qquad \mbox{for }\alpha>\Cq\lambda
\end{equation}	
since $q_{\alpha}(\lambda)\leq\Cq/\alpha<1/\lambda$. 
	This shows that $\|r_{\alpha}(T^*T)\xdag\|/\|R_\alpha\|$ tends to $0$ as 
	$\alpha\to 0$ and to $\infty$ as $\alpha\to \alphamax$. 
	Now $\deltaset(\xdag)=(0,\infty)$ follows from continuity and the intermediate 
	value theorem. 
	\item If $q_{\alpha}(\lambda)$ is not continuous or $\alphamax\neq\infty$, we still have the same 
	limiting behaviour of $\|r_{\alpha}(T^*T)\xdag\|/\|R_\alpha\|$ for 
	$\alpha\to 0$. However, 
	we have to exclude the case $r_{\alpha}(T^*T)\xdag=0$ for all $\alpha$ 
	in some neighborhood of $0$ to ensure that $0$ is a cluster point of 
	$\deltaset(\xdag)$. This is achieved by \eqref{eq:rneq0} and 
	the assumption $E_\alpha\xdag\neq 0$ for all $\alpha>0$. 
	\item For Landweber iteration and Lardy's method we have 
	\[
	\deltaset(\xdag)=\braces{\delta_n(\xdag):= 
	\frac{\|r_{1/n}(T^*T)\xdag\|}{\|R_{1/n}\|}: n\in\Nset}\qquad \mbox{and}\qquad 
	\gamma = \sup_{n\in\Nset}\frac{\delta_n(\xdag)}{\delta_{n+1}(\xdag)}.
	\]	
	Using \eqref{eq:Ralpha_norm} we can bound quotients of the denominators 
	of $\delta_n(\xdag)$ by 
	\begin{align*}
	\Tnorm{R_{1/(n+1)}}{} 
&= \sup_{\lambda\in\sigma(T^*T)} \frac{1-r_{1/(n+1)}(\lambda)}{\sqrt\lambda}
\leq \sup_{\lambda\in\sigma(T^*T)} \frac{1-r_{1/(n+1)}(\lambda)}{1-r_{1/n}(\lambda)}
\sup_{\lambda\in\sigma(T^*T)} \frac{1-r_{1/n}(\lambda)}{\sqrt\lambda}\\
&\leq \sup_{\lambda\in[0,\|T^*T\|]} \frac{1-r_{1/(n+1)}(\lambda)}{1-r_{1/n}(\lambda)}
\Tnorm{R_{1/n}}{}.
	\end{align*}
	Now setting $x=(1-\mu \lambda)$ and $x=\beta/(\beta+\lambda)$ for Landweber iteration and Lardy's method respectively we obtain the bound
	\begin{equation*}
		\frac{\Tnorm{R_{1/(n+1)}}{}}{\Tnorm{R_{1/n}}{}}
		\leq \sup_{x\in [0,1]} \frac{1-x^{n+1}}{1-x^n}
		=\sup_{x\in [0,1]} \rbra{x+\frac{1-x}{1-x^n}}
		\leq 2.
	\end{equation*}	
For Landweber iteration quotients of enumerators of $\delta_n(\xdag)$ are bounded by
	\begin{equation*}
		\frac{\Tnorm{r_{1/n}(T^*T)\xdag}{}^2}{\Tnorm{r_{1/(n+1)}(T^*T)\xdag}{}^2}
		=\frac{\int_0^{\Tnorm{T^*T}{}} (1-\mu\lambda)^{2n} \,\dd \Tnorm{E_\lambda \xdag}{}^2}{\int_0^{\Tnorm{T^*T}{}} (1-\mu\lambda)^{2n+2} \,\dd \Tnorm{E_\lambda \xdag}{}^2}
		\leq \frac{1}{\rbra{1-\mu \norm{T^*T}{}}^2}
	\end{equation*}
	since $(1-\mu\lambda)\geq 1-\mu \norm{T^*T}{}$ for all $\lambda\leq \norm{T^*T}{}$. Similar for Lardy's method we use $\beta/(\beta+\lambda)\geq\beta/(\beta+\norm{T^*T}{})$ for all $\lambda\leq \norm{T^*T}{}$ to obtain
	\begin{equation*}
		\frac{\Tnorm{r_{1/n}(T^*T)\xdag}{}^2}{\Tnorm{r_{1/(n+1)}(T^*T)\xdag}{}^2}
		=\frac{\int_0^{\Tnorm{T^*T}{}} \rbra{\frac{\beta}{\beta+\lambda}}^{2n} \,\dd \Tnorm{E_\lambda \xdag}{}^2}{\int_0^{\Tnorm{T^*T}{}} \rbra{\frac{\beta}{\beta+\lambda}}^{2n+2} \,\dd \Tnorm{E_\lambda \xdag}{}^2}
		\leq \rbra{1+\frac{\Tnorm{T^*T}{}}{\beta}}^2. 
	\end{equation*}
This shows that $\gamma$ is finite in both cases. \qedhere
\end{enumerate}
\end{proof}

The difference between our Theorem \ref{thm:det_noise_equiv} 
and the corresponding results in \cites{Neubauer1997,Albani2015} 
is analogous to the difference between the concepts of weakly and 
strongly quasioptimal parameter choice rules as introduced 
by Raus \& H\"amarik \cite{RH:07}. They called a parameter 
choice rule $\alpha_*:[0,\infty)\times \Yspace\to[0,\infty)$ 
\emph{weakly quasioptimal} (or simply quasioptimal) for the 
regularization method $\{R_{\alpha}\}$ if 
there exists a constant $C>0$ such that 
\begin{equation}\label{eq:quasioptimal}
\|R_{\alpha_*(\delta,\gobs)}\gobs-\xdag\|
\leq C\inf_{\alpha>0}\sup_{\|\xi\|\leq \delta}
\|R_{\alpha}(T\xdag+\xi)-\xdag\|+\mathcal{O}(\delta)
\end{equation}
for all $\xdag\in\Xspace$ and all $\gobs\in\Yspace$ 
with $\|\gobs-T\xdag\|\leq \delta$ as $\delta\to 0$. 
(Our formulation of this definition 
slightly differs from that in \cite{RH:07}, but it is equivalent 
due to the arguments at the beginning of part 2 of the proof of 
Theorem \ref{thm:det_noise_equiv}.) 
A parameter choice rule $\alpha_*$ is called 
\emph{strongly quasioptimal} if there exists $C>0$ such that 
\begin{equation}\label{eq:stronglyquasioptimal}
\|R_{\alpha_*(\delta,\gobs)}\gobs-\xdag\|
\leq C\sup_{\|\xi\|\leq \delta}\inf_{\alpha>0}\|R_{\alpha}(T\xdag+\xi)-\xdag\|+\mathcal{O}(\delta)
\end{equation}
for all $\xdag\in\Xspace$ and all $\gobs\in\Yspace$ with 
$\|\gobs-T\xdag\|\leq \delta$.
Lemma \ref{lem:infSupSwap} shows that the notions of weak and strong 
quasioptimality coincide for continuous regularization methods. 
In \cite{RH:07} it is shown that the discrepancy principle is strongly 
quasioptimal for regularization methods of infinite qualification such 
as Landweber iteration, Lardy's method or spectral cut-off (but 
it is not even weakly quasioptimal for Tikhonov regularization 
and iterated Tikhonov regularization). For iterated Tikhonov 
regularization the Raus-Gfrerer rule is weakly quasioptimal by 
results in \cite{raus:85}, and hence by Lemma \ref{lem:infSupSwap} also 
strongly quasioptimal. 
Moreover, Lepski{\u\i}'s rule was shown to be weakly 
quasioptimal for all considered methods (\cite{RH:07}), and hence 
it is strongly quasioptimal for (iterated) Tikhonov regularization,  
Showalter's method, and modified spectral cut-off  
by Lemma \ref{lem:infSupSwap}. 
The Monotone Error Rule is strongly quasioptimal for 
Landweber iteration and Lardy's method (\cite{RH:07}).  
In most cases the constant $C$ in \eqref{eq:quasioptimal} or 
\eqref{eq:stronglyquasioptimal} can be given explicitly.

\begin{theorem}\label{theo:det_noise_aposteriori}
Suppose Assumption \ref{ass:SR} holds true, let  $\alpha_*$ 
be weakly quasioptimal parameter choice rule, let $\psi_{\kappa}$ be concave, 
and assume that \eqref{eq:kappa_growth} holds true. 
Then the following statements are equivalent for any $\xdag\in\Xspace$ 
for which $\Delta(\xdag)$ satisfies \eqref{eq:gaps_deltaset}:
\begin{enumerate}
	\item\label{item:corRate} $\sup_{\alphamax\geq\alpha>0}\frac{1}{\kappa(\alpha)^2}\|r_{\alpha}(T^*T)\xdag\|^2<\infty$.
	\item\label{item:corWeak} For any finite $\delta_0>0$ we have
	\begin{equation*}
		\sup_{\delta\in(0,\delta_0]} \frac{1}{\psi_{\kappa}(\delta^2)} \sup_{\|\xi\|\leq\delta}\|R_{\alpha_*(\delta,T\xdag+\xi)}(T\xdag+\xi)-\xdag\|^2 <\infty.
	\end{equation*}
\end{enumerate}
\end{theorem}
\begin{proof}
	\emph{(\ref{item:corRate}) $\Rightarrow$ (\ref{item:corWeak}):} Using 
	Theorem \ref{thm:det_noise_equiv} and the definition of weak quasioptimality 
	\eqref{eq:quasioptimal} we see that there exists a constant $C$ such that for all $\delta>0$ the estimate
	\begin{equation*}
		\sup_{\snorm{\xi}{}\leq\delta} \snorm{R_{\alpha_*(\delta,T\xdag+\xi)}(T\xdag+\xi)-\xdag}{}^2 \leq C \rbra{\psi_{\kappa}(\delta^2) + \delta^2}
	\end{equation*}
	holds true. 
	Since $\psi_{\kappa}$ is concave $\lim_{t\rightarrow 0} t/\psi_{\kappa}(t)$ 
	is bounded and hence we obtain (\ref{item:corWeak}) for any finite $\delta_0>0$.
	
	\emph{(\ref{item:corWeak}) $\Rightarrow$ (\ref{item:corRate}):} 
	By Lemma \ref{lem:infSupSwap} we have
	\begin{align*}
	&\inf_{\alphamax\geq\alpha>0} \sup_{\snorm{\gobs-T\xdag}{}\leq\delta} 
	\snorm{R_{\alpha}(\gobs)-\xdag}{}^2 
	\leq 8\sup_{\snorm{\gobs-T\xdag}{}\leq\delta} \inf_{\alphamax\geq\alpha>0} 
	\snorm{R_{\alpha}(\gobs)-\xdag}{}^2 \\
	&\leq 8\sup_{\snorm{\gobs-T\xdag}{}\leq\delta} 
		\snorm{R_{\alpha_*(\delta,\gobs)}(\gobs)-\xdag}{}^2 
	\leq 8C(\xdag,\delta_0)\psi_{\kappa}(\delta)	
	\end{align*}
	for all $\delta\in \Delta(\xdag)\cap [0,\delta_0]$. 
	Now choose $\delta_0= \Theta_\kappa(2B\beta/(1-\Cdiag)^2)$ with $\beta=\lmin{\alphamax}{\Tnorm{T^*T}{}}$ and first 
	assume that $\Delta(\xdag)\cap(0,\delta_0] = (0,\delta_0]$. 
	Then we obtain 
	$\sup_{\alpha\in (0,\beta]}\frac{1}{\kappa(\alpha)^2}\|r_{\alpha}(T^*T)\xdag\|^2<\infty$ following the proof of Theorem \ref{thm:det_noise_equiv}. 
	As $\|r_{\alpha}(T^*T)\xdag\|^2\leq\|\xdag\|^2$ for all $\alpha>0$, 
	this is equivalent to (\ref{item:corRate}).
	
	If $\Delta(\xdag)$ has gaps, but satisfies \eqref{eq:gaps_deltaset}, 
	then for each $\delta\in (0,\delta_0]$ we can find 
	$\underline{\delta}\in [\delta/\gamma,\delta]$ with 
	$\underline{\delta}\in \deltaset(\xdag)$. 
	By the concavity of $\psi_{\kappa}$ we have 
	$\psi_{\kappa}(\delta)/\psi_{\kappa}(\underline{\delta})\leq \gamma$. 
	Therefore, replacing the supremum over $\delta\in(0,\delta_0]\cap\Delta(\xdag)$ 
	by the supremum over $\delta\in (0,\delta_0]$ increases the value at most 
	by a factor $\gamma$. 
\end{proof}

\section{Converse results for white noise}\label{sec:converse_white}
We now want to prove a theorem similar to Theorem \ref{thm:det_noise_equiv} for the white noise error model \eqref{eq:white_noise_model} using the expected square 
error as error measure.  By the bias-variance decomposition this equals 
\begin{equation}\label{eq:biasVarianceDecomposition}
	\EW{\norm{\widehat{\sol}_{\alpha}-\xdag}{\Xspace}^2}=\norm{\EW{\widehat{\sol}_{\alpha}}-\xdag}{\Xspace}^2 +\varepsilon^2 \EW{\norm{R_\alpha W}{}^2}.
\end{equation}
By Theorem \ref{thm:biasDecayEquivalence} the bias can be controlled by assuming that $\xdag\in \XTkappa$. 
The variance is given by $\varepsilon^2 \EW{\norm{R_\alpha W}{}}^2
= \varepsilon^2\mathrm{trace}(R_{\alpha}^*R_{\alpha})$, i.e.\ 
as opposed to the deterministic the effect of the noise is not described 
by the maximum, but by the sum of the eigenvalues of $R_{\alpha}^*R_{\alpha}$. 
Often the sum grows faster than the maximum as $\alpha\to 0$, and the specific 
rate depends not only on the regularization methods, but also on the eigenvalue 
distribution of the operator. We will assume that there exists a constant $D\geq 1$ 
and a monotonically decreasing function $v\in C((0,\infty))$ such that
\begin{subequations}\label{eqs:bounds_variance}
\begin{equation}\label{eq:bounds_variance}
			\frac{1}{D} v(\alpha)^2 \leq \EW{\norm{R_\alpha W}{}^2} \leq D v(\alpha)^2\qquad \forall \alphamax\geq\alpha>0
\end{equation}
with limits $\lim_{t\rightarrow 0} v(t)=\infty$ and  
$\lim_{t\rightarrow \infty} v(t)=0$. Note the in the deterministic case, i.e.\ 
for $\EW{\norm{R_\alpha W}{}^2}$ replaced by $\|R_{\alpha}\|^2$, we could simply 
choose $v(\alpha)= c/\sqrt{\alpha}$. 
Moreover, we will assume  that 
$v(\alpha)$ does not grow faster than polynomially as $\alpha\to 0$, 
or equivalently, that the inverse function 
$v^{-1}:(0,\infty)\to (0,\infty)$ does not decay faster than polynomially 
at infinity in the sense that there exists $p\geq 1$ such that 
\begin{equation}\label{eq:growth_vinv}
v^{-1}(rt)\geq r^{-q} v^{-1}\left(t\right)
\end{equation} 
for all $t>0$ and $r\geq 1$.
\end{subequations}

It was shown in \cite{BHMR:07} that 
$\EW{\|R_{\alpha}W\|^2}\sim \EW{\|(T(I-E_{\alpha}^{T^*T}))^{\dagger}\|^2}$ under 
certain conditions, and explicit expressions for $v$ have been derived.

\begin{theorem}\label{thm:white_noise_equiv}
	Let Assumption \ref{ass:SR} and \eqref{eqs:bounds_variance} hold true and define
\begin{equation*}
\psi_{\kappa,v} (t):=\kappa \rbra{\Theta^{-1}_{\kappa,v}\rbra{\sqrt{t}}}^2 
\qquad \text{with} \qquad \Theta_{\kappa,v} (\alpha):=\frac{\kappa(\alpha)}{v(\alpha)}.
\end{equation*}
Moreover, assume that $\kappa$ satisfies \eqref{eq:kappa_growth}.
Then for $\xdag\in\Xspace$ the following statements are equivalent:
	\begin{enumerate}
\item\label{item:wneSpace} $A:=\sup_{0<\alpha\leq \alphamax}\frac{1}{\kappa(\alpha)^2} \Tnorm{r_{\alpha}(T^*T)\xdag}{\Xspace}^2<\infty$ 
\item\label{item:wneExpectation} $B:=\sup_{0<\varepsilon\leq\Theta_{\kappa,v}(\alphamax)} \frac{1}{\psi_{\kappa,v}(\varepsilon^2)}\inf_{0<\alpha\leq\alphamax} \TEW{\Tnorm{R_{\alpha}(T\xdag+\varepsilon W)-\xdag}{\Xspace}^2}<\infty$.
\end{enumerate}
	More precisely,
\begin{equation*}
B\leq A+D\qquad \text{and} \qquad 
A\leq \lmax{\lmax{B}{B(2BD)^{pq}}}{\frac{\Tnorm{\xdag}{}^2}{\kappa \rbra{v^{-1}(\sqrt{2BD} v(\alphamax))}}}.
\end{equation*}
\end{theorem}

\begin{proof}
	\emph{(\ref{item:wneSpace}) $\Rightarrow$ (\ref{item:wneExpectation}):} 
	Set $\widehat{\sol}_{\alpha}:=R_{\alpha}(T\xdag+\varepsilon W)$. 
	By (\ref{item:wneSpace}) and Theorem \ref{thm:biasDecayEquivalence} we can bound 
	$\Tnorm{\TEW{\widehat{\sol}_{\alpha}}-\xdag}{\Xspace}^2
	= \Tnorm{r_{\alpha}(T^*T)\xdag}{\Xspace}^2	\leq A \kappa(\alpha)^2$, 
	and by assumption \eqref{eq:bounds_variance} we can estimate 
	$\TEW{\Tnorm{R_\alpha W}{}^2}\leq D v^2(\alpha)$. Hence,
	\begin{equation*}
		\EW{\norm{\widehat{\sol}_{\alpha}-\xdag}{\Xspace}^2}\leq A \kappa^2(\alpha)+ D \varepsilon^2 v^2(\alpha).
	\end{equation*}
The minimum over the right hand side is approximately attained if 
$\kappa(\alpha)=\varepsilon v(\alpha)$ or equivalently if 
$\alpha=\Theta^{-1}_{\kappa,v}(\varepsilon)$. The equality 
$\kappa(\alpha)=\varepsilon v(\alpha)$ implies in particular that  
\begin{equation}\label{eq:def_psi_kappav_alt}
	\psi_{\kappa,v}(\varepsilon^2)=\varepsilon^2 v (\Theta^{-1}_{\kappa,v}(\varepsilon))^2.
\end{equation}	
Therefore, for all $\varepsilon>0$  we obtain
\begin{equation*}
		\inf_{0<\alpha\leq\alphamax} \EW{\norm{\widehat{\sol}_{\alpha}-\xdag}{\Xspace}^2}\leq \sbra{A+D} \kappa \rbra{\Theta^{-1}_{\kappa,v}(\varepsilon)}^2
	\end{equation*}
and can choose $B=A+D$.
	
	\emph{(\ref{item:wneExpectation}) $\Rightarrow$ (\ref{item:wneSpace}):} Using again \eqref{eq:biasVarianceDecomposition} and the lower bound on the variance 
in \eqref{eq:bounds_variance} we obtain
	\begin{equation*}
		\EW{\norm{\widehat{\sol}_{\alpha}-\xdag}{\Xspace}^2}\geq \norm{\EW{\widehat{\sol}_{\alpha}}-\xdag}{\Xspace}^2 +\frac{1}{D} v(\alpha)= \norm{r_{\alpha}(T^*T)\xdag}{\Xspace}^2+\frac{\varepsilon^2}{D} v(\alpha)^2.
	\end{equation*}
	Because the first term is increasing and the second term is decreasing in 
	$\alpha$,  we obtain 
\begin{align}\label{eq:aux_white_noise}
B \psi_{\kappa,v}(\varepsilon^2) \geq& 
		\inf_{\alpha>0}\bracket{\norm{r_{\alpha}(T^*T)\xdag}{\Xspace}^2+\frac{\varepsilon^2}{D} v(\alpha)^2}
		\geq \lmin{\norm{r_{\alpha_*}(T^*T)\xdag}{\Xspace}^2}
{\frac{\varepsilon^2}{D} v(\alpha_*)^2}.
\end{align}
for any $\alpha_*\in(0, \alphamax]$. We will choose 
	\begin{equation}\label{eq:wneAlphaStar}
		\alpha_*(\varepsilon)=v^{-1}\rbra{\sqrt{2BD}\, v \rbra{\Theta^{-1}_{\kappa,v}\rbra{\varepsilon}}}.
	\end{equation}
Using \eqref{eq:def_psi_kappav_alt} we see that the second term in the 
	minimum in \eqref{eq:aux_white_noise} equals twice the left hand side: 
	\[
	\frac{\varepsilon^2}{D} v(\alpha_*)^2 
	= 2B\varepsilon^2 v (\Theta^{-1}_{\kappa,v}(\varepsilon))^2
	= 2B \psi_{\kappa,v}(\varepsilon^2), 
	\]
	Therefore, $\norm{r_{\alpha_*}(T^*T)\xdag}{\Xspace}^2\geq\frac{\varepsilon^2}{D} v(\alpha_*)^2$ leads to a contradiction, i.e.\ the minimum 
in \eqref{eq:aux_white_noise} is attained at the first argument. We obtain 
	\begin{align*}
		\norm{r_{\alpha_*}(T^*T)\xdag}{\Xspace}^2
		\leq& B \psi_{\kappa,v}(\varepsilon^2) = 
		B \psi_{\kappa,v}\rbra{\rbra{\Theta_{\kappa,v}\rbra{v^{-1}\rbra{\frac{v(\alpha_*)}{\sqrt{BD}}}}}^2},
		\end{align*}
where we have solved \eqref{eq:wneAlphaStar} for $\varepsilon$ in the 
second step. 
Abbreviating $z:=v^{-1}\rbra{v(\alpha_*)/\sqrt{2BD}}$ and using 
\eqref{eq:def_psi_kappav_alt} again, we find 
		\begin{align*}
&\norm{r_{\alpha_*}(T^*T)\xdag}{\Xspace}^2
\leq  B \psi_{\kappa,v}\rbra{\rbra{\Theta_{\kappa,v}\rbra{z}}}^2 = 
B \rbra{\Theta_{\kappa,v}\rbra{z}}^2 v\rbra{\Theta^{-1}_{\kappa,v} \rbra{\Theta_{\kappa,v}\rbra{z}}}^2\\
\qquad\qquad&= B\kappa\rbra{z}^2
		\leq B \kappa\rbra{\paren{\lmax{1}{(2BD)^{q/2}}} \alpha_*}^2
		\leq \paren{\lmax{B}{B(2BD)^{pq}}} \kappa(\alpha_*)^2
	\end{align*}
	for all $\alpha\in (0,\alphamax]$ defined by \eqref{eq:wneAlphaStar} using \eqref{eq:growth_vinv}.
	
	For $\alpha\in(0,\alphamax]$ not of the given form note that $\alpha\geq v^{-1}(\sqrt{2BD} v(\alphamax))$ and using mononicity we obtain for these $\alpha$
	\begin{equation*}
		\frac{1}{\kappa(\alpha)^2}\norm{r_{\alpha}(T^*T)\xdag}{\Xspace}^2\leq \frac{\Tnorm{\xdag}{}^2}{\kappa \rbra{v^{-1}(\sqrt{2BD} v(\alphamax))}}
	\end{equation*}
	showing boundedness for all $\alpha\in(0,\alphamax]$.
\end{proof}

\begin{remark}\label{rem:white_noise_equiv} 
If assumption \eqref{eq:bounds_variance} is relaxed to 
\begin{equation}\label{eq:bounds_variance_relax}
			 v_-(\alpha)^2 \leq \EW{\norm{R_\alpha W}{}^2} \leq v_+(\alpha)^2\qquad \forall \alphamax\geq\alpha>0
\end{equation}
where possibly $\lim_{\alpha\to 0}(v_+/v_-)(\alpha)=\infty$, and 
$v_-$ satisfies \eqref{eq:growth_vinv}, then it can be seen by inspection of the proof that 
\begin{align*}
	\text{Theorem \ref{thm:white_noise_equiv}(\ref{item:wneSpace}) } &\Rightarrow\sup_{0<\varepsilon\leq  \Theta_{\kappa,v_+}(\alphamax)} \frac{1}{\psi_{\kappa,v_+}(\varepsilon^2)}\inf_{0<\alpha \leq \alphamax} \EW{\Tnorm{R_{\alpha}(T\xdag+\varepsilon W)-\xdag}{\Xspace}^2}<\infty,\\
	\text{Theorem \ref{thm:white_noise_equiv}(\ref{item:wneSpace}) } &\Leftarrow\sup_{0<\varepsilon\leq  \Theta_{\kappa,v_-}(\alphamax)} \frac{1}{\psi_{\kappa,v_-}(\varepsilon^2)}\inf_{0<\alpha \leq \alphamax} \EW{\Tnorm{R_{\alpha}(T\xdag+\varepsilon W)-\xdag}{\Xspace}^2}<\infty.
\end{align*}

This is relevant for operators $T$ with exponentially decaying singular values. 
Whereas for polynomial decay assumption \eqref{eq:bounds_variance} can 
be verified using results from \cite{BHMR:07}, for singular values 
with asympototic behaviour $\sigma_j(T)\sim \exp(-cj^{\beta})$ with 
$c,\beta>0$ one can only (easily) verify the relaxed condition 
\eqref{eq:bounds_variance_relax} with 
\[
v_-(\alpha) = c_-\alpha^{-1/2}\qquad \mbox{and}\qquad 
v_+(\alpha) = c_+\alpha^{-1/2-\tau}
\]
for any $\tau>0$ and some $c_-,c_+>0$. However, for such operators 
(\ref{item:wneSpace}) is typically satisfied only for
logarithmic functions $\kappa(\alpha) = (-\ln \alpha)^{-p}$ 
with some $p>0$ for $\xdag$ of finite smoothness. 
In this case one has 
\[
\psi_{\kappa,v_+}(t) = (-\ln t)^{-2p}(1+o(1)), 
\qquad t\to 0
\] 
independent of the choice of $\tau\in [0,\infty)$ (see \cite{mair:94}). 
Therefore, the equivalence in Theorem \ref{thm:white_noise_equiv} still 
holds true with either $v=v_-$ or $v=v_+$. 
\end{remark}

\section{Besov spaces as maxisets}\label{sec:spaceEquivalence}
We have seen in the previous sections that convergence rates of $\psi_\kappa$ to 
a true solution $\xdag$ for regularization methods are completely characterized by 
$\xdag\in\XTkappa$. Andreev \cite{Andreev2015a} showed that these spaces coincide 
with K-interpolation spaces with equivalent norms. Recall that for a 
Banach space $\Zspace\subset \Xspace$, which is continuously embedded 
in $\Xspace$ the K-functional is defined by 
\[
K_t(\sol):= \inf_{g\in \Zspace}\paren{\|\sol-g\|_{\Xspace}^2
+t^2\|g\|_{\Zspace}^2}^{1/2}.
\]
For $\nu\in (0,1)$ the K-interpolation space with fine index $\infty$ is defined by 
\[
(\Xspace,\Zspace)_{\nu,\infty}:=
\{\sol\in \Xspace:\altnorm{\sol}{(\Xspace,\Zspace)_{\nu,\infty}}<\infty\}
\quad \mbox{where}\quad 
\altnorm{\sol}{(\Xspace,\Zspace)_{\nu,\infty}}:=\sup_{t>0} t^{-\nu}K_t(\sol).
\]
Here we temporarily switch to a different norm notation because of the 
numereous indices. 
It can be shown that $(\Xspace,\Zspace)_{\nu,\infty}$ with this norm 
is a Banach space. 
If $\Zspace=(S^*S)^k(\Xspace)$
for a bounded linear operator $S:\Xspace\to\Yspace$ and 
some $k\in \Nset$ with norm $\norm{\sol}{\Zspace}:=\|(S^*S)^{-k}\sol\|_{\Xspace}$, 
Andreev \cite{Andreev2015a} showed that 
\begin{equation}\label{eq:andreev_equiv}
\Xspace^S_{\mathrm{id}^{k\nu}}=(\Xspace,(S^*S)^k(\Xspace))_{\nu,\infty}
\end{equation}
for $\nu\in (0,1)$ with 
\[
\sqrt{1-\nu}\altnorm{\sol}{(\Xspace,(S^*S)^k(\Xspace)_{\nu,\infty}} 
\leq \altnorm{\sol}{\Xspace^S_{\mathrm{id^{k\nu}}}}
\leq \rbra{1\!-\!\nu}^{1-\nu} \nu^{\nu} 
\altnorm{\sol}{(\Xspace,(S^*S)^k(\Xspace)_{\nu,\infty}}. 
\]

We further recall that the K-interpolation of certain Sobolev spaces yields 
Besov spaces. In particular,  
\begin{equation}\label{eq:besov_interp}
(L^2(\manifold),H^k(\manifold))_{\nu,\infty} = B^{k\nu}_{2,\infty}(\manifold)
\end{equation}
if $\manifold$ is a smooth Riemannian manifold with Laplace-Beltrami operator 
$\Delta$ satisfying Assumption \ref{ass:manifold} below 
and $H^k(\manifold):=(I-\Delta)^{-k/2}(L^2(\manifold))$ with 
norm $\|\sol\|_{H^k}:=\|(I-\Delta)^{k/2}\sol\|_{L^2}$ 
(see \cite[Chapter~7]{Triebel1992}). 
\begin{assumption}\label{ass:manifold}
	Let $\manifold$ be a connected smooth Riemanian manifold. Let $\manifold$
	\begin{itemize}
		\item be complete,
		\item have an injectivity radius $r>0$ and
		\item a bounded geometry.
	\end{itemize}
\end{assumption}
Here completeness means that all geodesics are infinitely extendable, the 
injectivity radius refers to the size of the domains in which the exponential 
map is bijective, and bounded geometry means that the determinant of the 
Riemannian metric is bounded from below by a positive constant and all its 
derivatives are bounded from above (see \cite{Triebel1992} for further 
discussions). 
Important examples of such manifolds include $\Rset^n$ and compact manifolds 
without boundaries. 
In the following we will consider operators $T\colon \Xspace=L^2(\manifold)\rightarrow \Yspace$ such that 
\begin{equation}\label{eq:T_Laplace}
	T^*T=\phiOp(-\Delta),
\end{equation}
where $\phiOp$ fulfills the following conditions:
\begin{assumption}\label{ass:operator}
	Let $\phiOp:[0,\infty)\rightarrow(0,\infty)$ such that
	\begin{itemize}
		\item $\phiOp$ is continuous,
		\item $\phiOp|_{[t_0,\infty)}$ is strictly decreasing for some 
		$t_0\geq 0$,
		\item $\phiOp(\mu)\rightarrow 0$ for $\mu \rightarrow \infty$. 
	\end{itemize}
\end{assumption}

Our aim of this section is to prove the following theorem:
\begin{theorem}\label{thm:normEquivalence}
	Let $\manifold$ fulfill Assumption \ref{ass:manifold}, $\phiOp$ fulfill Assumption \ref{ass:operator} and $s>0$. Let $T:L^2(\manifold)\rightarrow \Yspace$ be of the form \eqref{eq:T_Laplace} and define 
	\begin{equation*}
		\kappa(\alpha) :=
		\begin{cases} 
		0,& \text{if }\alpha=0\\
		\rbra{{\phiOp|_{[t_0,\infty)}}^{-1}(\alpha)}^{-1/2},&\text{if } \alpha\in ((0,\phiOp(t_0)]),\\
		t_0^{-1/2},&\text{if }\alpha>\phiOp(t_0).
		\end{cases}
	\end{equation*}
	Then $\Xspace^T_{\kappa^s}=\Bs$ with equivalent norms.
\end{theorem}


\begin{proof}
We introduce the operator $S:=\kappa(T^*T)^{1/2}:L^2(\manifold)\rightarrow L^2(\manifold))$ such that 
\begin{equation*}
	S^*S= \kappa(T^*T)=(\kappa\circ\Lambda)(-\Delta). 
\end{equation*}
As $(\kappa\circ\Lambda)(t) = t^{-1/2}$ for $t\geq t_0$ and 
$\inf_{0\leq t\leq t_0}(\kappa\circ\Lambda)(t)>0$ by continuity, 
we have  
\begin{equation}\label{eq:SS_Hk}
(S^*S)^k(L^2(\manifold))=H^k(\manifold)
\end{equation} 
with equivalent norms for all $k\in\Nset$. 
Using the substitution $t=\kappa(\alpha)$ we obtain
	\begin{align*}
		\altnorm{\sol}{\Xspace^S_{\mathrm{id}^s}}&=\sup_{0<t\leq 1/t_0} t^{-s} \norm{E^{S^*S}_{t} \sol}{}
		= \sup_{\alpha\in (0,\phiOp(t_0)]} \frac{1}{\kappa(\alpha)^s} \norm{E^{S^*S}_{\kappa(\alpha)} \sol}{}
		=\sup_{\alpha\in (0,\phiOp(t_0)]} \frac{1}{\kappa(\alpha)^s} \norm{E^{T^*T}_{\alpha} \sol}{}.
	\end{align*}
As $E^{T^*T}_{\alpha}(\sol)=E^{T^*T}_{\Lambda(0)}(\sol)=\sol$ for 
	$\alpha>\Lambda(0)$, this shows that the norms 
$\altnorm{f}{\Xspace^S_{\mathrm{id}^s}}$ and $\altnorm{\sol}{\Xspace^T_{\kappa^s}}$ 
are equivalent. 
	Choosing $k\in\Nset$ with $k>s$ and using \eqref{eq:andreev_equiv}, 
	\eqref{eq:SS_Hk} and \eqref{eq:besov_interp} we obtain
	\[
	\Xspace^T_{\kappa^s} = \Xspace^S_{\mathrm{id}^s} 
	= (L^2(\manifold),(S^*S)^k(\manifold))_{s/k,\infty}
	= B^s_{2,\infty}(\manifold) 
\]	
with equivalent norms.
\end{proof}

\section{Examples}\label{sec:examples}
In this section we want to apply our results to some examples. The examples are taken from \cite{hohage:00} and complement the results there.

\subsection{Operators in Sobolev scales}
In the following  we describe a fairly general class of problems. It contains 
convolution operators (if $\manifold =\Rset^d$ or $\manifold=(\mathbb{S}^1)^d$), for which the convolution kernel has a certain 
type of singularity at $0$, boundary integral operators,  injective elliptic pseudo-differential operators, and compositions 
of such operators. 
\begin{theorem}\label{thm:sobo_scales}
Let $\manifold$ be a $d$-dimensional manifold satisfying Assumption \ref{ass:manifold}, and 
let $T$ be an operator which is $a$ times smoothing ($a>d/2$) in the sense 
that $T:H^s(\manifold)\to H^{s+a}(\manifold)$ is well-defined, bounded and 
has a bounded inverse for all $s\in\Rset$. We will consider $T$ as an operator 
from $L^2(\manifold)$ into itself, i.e.\ $\Xspace=\Yspace = L^2(\manifold)$. 
We consider a spectral regularization method with classical qualification 
$\mu_0\geq 1$ satisfying Assumption \ref{ass:SR}.
Then the following statements are equivalent for all 
$\xdag\in \Xspace\setminus\{0\}$ and $u\in (0,a)$: 
\begin{enumerate}
	\item\label{it:sobo_VSC} $\xdag$ satisfies a the VSC \eqref{eq:VSC_innerprod} with 
	$\psi(t) = Ct^{\frac{u}{u+a}}$ for some $C>0$. 
	\item\label{it:sobo_Besov} $\xdag \in B^{u}_{2,\infty}(\manifold)$.
\item\label{it:sobo_det_rate} 
For a quasioptimal parameter choice rule $\alpha_*$ 
and a  regularization method for which $\Delta(\xdag)$ satisfies \eqref{eq:gaps_deltaset} we have
	\[
	\sup\{\norm{R_{\alpha_*(\delta,\gobs)}\gobs-\xdag}{L^2}:\norm{\gobs-T\xdag}{L^2}\leq \delta\} 
	= \mathcal{O}\rbra{\delta^{\frac{u}{u+a}}},\qquad \delta\to 0.
	\]
\item\label{it:sobo_stoch_rate}
$\rbra{\inf_{\alpha>0} \EW{\norm{R_{\alpha}(T\xdag+\varepsilon W)-\xdag}{L^2}^2}}^{1/2}
  = \mathcal{O}\paren{\varepsilon^{\frac{u}{u+a+d/2}}},\qquad \varepsilon\to 0$.
\end{enumerate}
(\ref{it:sobo_Besov})--(\ref{it:sobo_stoch_rate}) are equivalent for all 
$u\in (0,2a\mu_0)$, and the assumption $a>d/2$ can be relaxed to $a>0$ if (\ref{it:sobo_stoch_rate}) is neglected. 
\end{theorem}

\begin{proof}
\emph{(\ref{it:sobo_VSC}) $\Leftrightarrow \xdag\in \Xspace_{\kappa}^T$:} 
Note that $\psi=\psi_{\kappa}$ with 
\[
\kappa(t) = C't^{u/2a}\qquad\mbox{ for some }C'>0,
\]
and that the assumption $u\in (0,a)$ ensures that $\kappa$ satisfies the conditions 
of Theorem \ref{thm:equiv_vsc_decay}. 

\emph{$\xdag\in \Xspace_{\kappa}^T\Leftrightarrow$ (\ref{it:sobo_Besov}):} 
It follows from \eqref{eq:andreev_equiv} and \eqref{eq:besov_interp} that 
\[
\Xspace^T_{\kappa}=
(L^2(\manifold),(T^*T)(L^2(\manifold)))_{u/2a,\infty} 
= (L^2(\manifold),H^{2a}(\manifold))_{u/2a,\infty} 
= B^{u}_{2,\infty}(\manifold).
\]

\emph{$\xdag\in \Xspace_{\kappa}^T\Leftrightarrow$ \eqref{eq:bias_sobo} below}: 
For $u/2a<\mu_0$ Theorem \ref{thm:biasDecayEquivalence} yields equivalence to
\begin{equation}\label{eq:bias_sobo}
\sup_{\alpha>0} \alpha^{-u/2a}\|r_{\alpha}(T^*T)\xdag\|<\infty.
\end{equation}

\emph{\eqref{eq:bias_sobo} $\Leftrightarrow$ (\ref{it:sobo_det_rate}):}
This follows from Theorem \ref{theo:det_noise_aposteriori}. 

\emph{\eqref{eq:bias_sobo} $\Leftrightarrow$ (\ref{it:sobo_stoch_rate}):}
It has been shown in \cite[\S 5.3]{BHMR:07} that \eqref{eq:bounds_variance} 
holds true with $v(\alpha) = \alpha^{-(a+d/2)/(2a)}$. Hence, we can apply 
Theorem \ref{thm:white_noise_equiv}.  
\end{proof}

\begin{example}
We consider a circle $\manifold =r\mathbb{S}^1\subset\Rset^2$ with 
$r>0$ and the single layer potential operator 
$(T\sol)(x):=-\frac{1}{\pi}\int_\manifold \ln|x-y|\sol(y)\,ds(y)$. 
Let $\sol_n(r\cos t,r\sin t):= (2\pi r)^{-1/2}\exp(int)$, $n\in\Zset$ denote 
the trigonometric basis of $L^2(r\mathbb{S}$). It is known 
(see \cite[Sec.~3.3]{kirsch:96}) that $T\sol_n= -1/|n|\sol_n$ for $n\neq 0$, 
and $T\sol_0 = \ln(r)\sol_0$. Let us choose $r=\exp(1)$ for simplicity. 
Recall that an (equivalent) norm on 
$H^s(\manifold)$ is given by $\Tnorm{\sol}{H^s}^2 =  
\sum_{n\in\Zset}(1\lor|n|)^{2s}\langle \sol,\sol_n\rangle^2$ for $s\geq 0$. 
W.r.t.\ this norm $T^{\nu}=(T^*T)^{u/2}$ is isometric from 
$H^s(\manifold)$ to $H^{s+u}(\manifold)$ for all $u>0$, so the 
assumptions of Theorem \ref{thm:sobo_scales} hold true with $a=d=1$. 
Moreover, 
the spectral source condition $\xdag\in \mathrm{ran}((T^*T)^{u/2}$ 
is equivalent to $\xdag\in H^u(\manifold)$ and yields the convergence 
rate $\mathcal{O}\big(\delta^{u/(u+1)}\big)$. 
The (equivalent) $\Xspace^T_{\kappa}$-norm of $B^{u}_{2,\infty}(\manifold)$ 
(with $\kappa(t)=t^u$) is given by 
$\Tnorm{\xdag}{B^u_{2,\infty}}^2=
\sup_{m\geq 0}(1\lor m)^{2u}\sum_{|n|\geq m}\langle\xdag,\sol_n\rangle^2$. 
This shows that $B^u_{2,\infty}(\manifold)$ is the set of 
$\sol\in L^2(\manifold)$ for which the $L^2$-orthogonal projections 
onto the space of trigonometric polynomials of degree $\leq m$ converge 
with rate $\mathcal{O}(m^{-u})$ as $m\to\infty$. Note that 
\[
\xdag = \sum_{n\in\Zset} (1\lor |n|)^{-u}\sol_n\in B^{u}_{2,\infty}(\manifold)\setminus H^u(\manifold)
\]
for any $u>0$, but $\xdag\in H^{\nu}(\manifold)$ for $\nu<u$. 
Therefore, we obtain the  convergence rate 
$\mathcal{O}(\delta^{u/(u+1)})$ for $\xdag$, whereas an analysis 
via spectral source conditions only yields rates 
$\mathcal{O}(\delta^{\nu/(\nu+1)})$ for $\nu\in(0,u)$. 
Moreover, as $\lim_{\nu\nearrow u}\snorm{\xdag}{H^{\nu}}=\infty$, constants explode 
as $\nu\to u$. 
\end{example}

\subsection{Backward heat equation}
Let us consider the heat equation on a manifold $\manifold$ satisfying Assumption \ref{ass:manifold}:
\begin{align*}
&\partial_t u = \Delta u &&\mbox{in }\manifold\times (0,\overline{t})\\
&u(\cdot,0) = \sol &&\mbox{on }\manifold
\end{align*}
The backward heat equation is the inverse problem to estimate the 
initial temperature $\sol$ from observations of the final 
temperature $\data = u(\cdot,\overline{t})$. This fits into the framework 
\eqref{eq:T_Laplace} with the function 
\begin{equation*}
\Lambda_{\rm BH}(\mu) = \exp(-2\overline{t}\mu).
\end{equation*}
We obtain the following equivalence result:
\begin{theorem}\label{thm:backwardHeat}
Let $\manifold$ be a compact manifold satisfying Assumption \ref{ass:manifold}. 
For spectral regularization methods satisfying Assumption (\ref{ass:SR}) 
and the forward operator 
$T:L^2(\manifold)\to L^2(\manifold)$ with $T^*T=\Lambda_{\rm BH}(-\Delta)$ of the 
backward heat equation the following statements for $\beta>0$ and 
$\xdag\in L^2(\manifold)\setminus\{0\}$ are equivalent:
\begin{enumerate}
\item\label{item:BH_Besov} $\xdag \in B^{2\beta}_{2,\infty}(\manifold)$.
\item\label{item:BH_VSC} $\xdag$ satisfies a VSC \eqref{eq:VSC_innerprod} 
with index function $\psi(t) = C\log(3+t^{-1})^{-2\beta}$ for some 
$C>0$.
\item\label{item:BH_quasi} For a quasioptimal parameter choice rule $\alpha_*$ and a
regularization method for which $\Delta(\xdag)$ satisfies \eqref{eq:gaps_deltaset} we have
\[
\sup\{\norm{R_{\alpha_*(\delta,\gobs)}\gobs-\xdag}{L^2}:\norm{\gobs-T\xdag}{L^2}\leq \delta\}
= \mathcal{O}\paren{\log(\delta^{-1})^{-\beta}},\qquad \delta\to 0.
\]
\item\label{item:BH_stoch}
$\left(\inf_{\alpha>0}\EW{\norm{R_{\alpha}(T\xdag+\varepsilon W)-\xdag}{L^2}^2}\right)^{1/2}
= \mathcal{O}\paren{\log(\varepsilon^{-1})^{-\beta}},\qquad \varepsilon\to 0.
$
\end{enumerate}
\end{theorem}

\begin{proof}
\emph{(\ref{item:BH_Besov}) $\Leftrightarrow\xdag\in \Xspace_{\kappa^{2\beta}}^T$:}
By Theorem \ref{thm:normEquivalence} we have $\xdag\in B^{2\beta}_{2,\infty}(\manifold)$ 
if and only if $\xdag\in\Xspace^T_{\kappa^{2\beta}}$ with
$\kappa(\alpha) =\rbra{(1/2\overline{t})\ln(\alpha^{-1})}^{-1/2}$ for 
$0<\alpha\leq\Lambda_{\rm BH}(t_0)$ and any $t_0>0$. 

\emph{$\xdag\in \Xspace_{\kappa^{2\beta}}^T \Leftrightarrow$ (\ref{item:BH_VSC}):} 
This follows from Theorem \ref{thm:equiv_vsc_decay} since 
\begin{equation}\label{eq:psi_BH}
	\psi_{\kappa^{2\beta}} (t)= C\log(t^{-1})^{-2\beta} \rbra{1+o(1)},\qquad \text{as } t \rightarrow 0
\end{equation}
as shown in \cite{mair:94}. The $3$ is included in the definition of $\psi$ 
to avoid a singularity at $t=1$. 

\emph{$\xdag\in \Xspace_{\kappa^{2\beta}}^T \Leftrightarrow$ (\ref{item:BH_quasi}):} Follows from Theorems \ref{thm:biasDecayEquivalence} and \ref{theo:det_noise_aposteriori}.

\emph{$\xdag\in \Xspace_{\kappa^{2\beta}}^T \Leftrightarrow$ (\ref{item:BH_stoch}):} Use the results of \cite[\S 5.1]{BHMR:07} to see that \eqref{eq:bounds_variance_relax} is fulfilled for any $\tau>0$ and apply Remark \ref{rem:white_noise_equiv} 
and Theorem \ref{thm:biasDecayEquivalence}.
\end{proof}

\subsection{Sideways heat equation}
We now consider the heat equation in the interval 
$[0,1]$. We may think of $[0,1]$ as the wall 
of a furnace where the right boundary $1$ is the 
inaccessible interior side and $0$ the accessible outer side. We assume the left boundary is insulated and impose 
the no-flux boundary condition $\partial_xu(0,t)=0$. 
The forward problem reads
\begin{align*}
&u_t = u_{xx} &&\mbox{in }[0,1]\times \Rset,\\
&u(1,t) = \sol(t),  &&t\in\Rset,\\
&u_x(0,t)=0,&&t\in\Rset.
\end{align*}
We will consider the inverse problem to estimate the 
temperature $\sol(t)=u(1,t)$ at the inaccessible side 
from measurements of the temperature $\data(t)=u(0,t)$ 
at the accessible side for all times $t\in\Rset$. 
As shown in \cite{hohage:00} this fits into the framework 
\eqref{eq:T_Laplace} if we set
\[
\Lambda_{\rm SH}(\mu) = \absval{\cosh\sqrt{i\sqrt{\mu}}}^{-2},
\qquad \manifold=\Rset.
\]

\begin{theorem}
For spectral regularization methods satisfying Assumption (\ref{ass:SR}) 
and the forward operator 
$T:L^2(\Rset)\to L^2(\Rset)$ such that $T^*T=\Lambda_{\rm SH}(-\Delta)$ of the 
sideways heat equation the following statements for $\beta>0$ and 
$\xdag\in L^2(\Rset)\setminus\{0\}$ are equivalent:
\begin{enumerate}
\item\label{item:SH_Besov} $\xdag \in B^{\beta/2}_{2,\infty}(\Rset)$.
\item\label{item:SH_VSC} $\xdag$ satisfies a VSC \eqref{eq:VSC_innerprod} 
with index function $\psi(t) = C\log(3+t^{-1})^{-2\beta}$ for some 
$C>0$.
\item\label{item:SH_quasi} For a quasioptimal parameter choice rule $\alpha_*$ and a
regularization method for which $\Delta(\xdag)$ satisfies \eqref{eq:gaps_deltaset} we have
\[
\sup\{\norm{R_{\alpha_*(\delta,\gobs)}\gobs-\xdag}{L^2}:\norm{\gobs-T\xdag}{L^2}\leq \delta\}
= \mathcal{O}\paren{\log(\delta^{-1})^{-\beta}},\qquad \delta\to 0.
\]
\item\label{item:SH_stoch}
 $
\left(\inf_{\alpha>0}\EW{\norm{R_{\alpha}(T\xdag+\varepsilon W)-\xdag}{L^2}^2}\right)^{1/2}
= \mathcal{O}\paren{\log(\varepsilon^{-1})^{-\beta}},\qquad \varepsilon\to 0.
$
\end{enumerate}
\end{theorem}
\begin{proof}
	\emph{(\ref{item:SH_Besov}) $\Leftrightarrow\xdag\in \Xspace_{\kappa^{\beta/2}}^T$:} As shown in \cite{hohage:00} $\Lambda_{\rm SH}(\mu)=(1/4)\exp(-\sqrt{2}\mu^{1/4})(1+o(\mu))$ as $\mu \rightarrow \infty$. Therefore we obtain 
	$\kappa(\alpha) = 2 \ln(\alpha^{-1})^{-2}(1+o(\alpha))$ as $\alpha \rightarrow 0$.

	\emph{$\xdag\in \Xspace_{\kappa^{\beta/2}}^T\Leftrightarrow$ (\ref{item:SH_VSC}) $\Leftrightarrow$ (\ref{item:SH_quasi}) 
	$\Leftrightarrow$ (\ref{item:SH_stoch}):} 
	This follows as in proof of Theorem \ref{thm:backwardHeat}. Due to the different 
	exponent in the asymptotic formula for $\kappa$ we have 
	$\psi_{\kappa^{\beta/2}} (t)= C\log(t^{-1})^{-2\beta} \rbra{1+o(1)}$ here 
	instead of \eqref{eq:psi_BH}.
\end{proof}

\subsection{Satellite gradiometry}
Let us assume that the Earth is a perfect ball of radius $1$. 
The gravitational potential $u$ of the Earth is determined by 
its values $\sol$ at the surface by the exterior boundary value 
problem 
\begin{align*}
&\Delta u = 0 &&\mbox{in }\{x\in\Rset^3: |x|>1\}\\
&|u|\to 0,  &&|x|\to \infty\\
&u=f&&\mbox{on }\mathbb{S}^2
\end{align*}
In satellite gradiometry one studies the inverse problem to determine 
$\sol$ from satellite measurements of the rate of change of the gravitational 
force in radial direction at height $R>0$, i.e.\ the data are described by 
the function $\data= \diffq{^2u}{r^2}|_{R\mathbb{S}^2}$. 
As shown in \cite{hohage:00} this fits into the 
framework \eqref{eq:T_Laplace} if we set 
\[
\phiOp_{\mathrm{SG}}(\mu) := \paren{\frac{1}{2}+\lambda}^2\paren{\frac{3}{2}+\lambda}^2R^{-2\lambda}, \qquad 
\lambda = \sqrt{\frac{1}{2}+\mu}, \qquad 
\manifold=\mathbb{S}^2.
\]
Note that $\phiOp_{\mathrm{SG}}$ (unlike $\phiOp_{\mathrm{BH}}$ and 
$\phiOp_{\mathrm{SH}}$) is not globally monotonically decreasing unless 
$R$ is large enough (one needs $R\geq \exp((4\sqrt 2 +2)/(\sqrt 2+5))\approx 3.3$, 
which is not realistic).
\begin{theorem}
For spectral regularization methods satisfying Assumption \ref{ass:SR}
and the forward operator 
$T:L^2(\mathbb{S}^2)\to L^2(\mathbb{S}^2)$ such that $T^*T =\Lambda_{\rm SG}(-\Delta)$ with $R$ large enough such that $\Lambda_{\rm SG}$ fulfills Assumption \ref{ass:operator} of the 
satellite gradiometry problem the following statements for $\beta>0$ and 
$\xdag\in L^2(\mathbb{S}^2)\setminus\{0\}$ are equivalent:
\begin{enumerate}
\item\label{item:SG_Besov} $\xdag \in B^{\beta}_{2,\infty}(\mathbb{S}^2)$.
\item\label{item:SG_VSC} $\xdag$ satisfies a VSC \eqref{eq:VSC_innerprod} with index 
function $\psi(t) = C\log(3+t^{-1})^{-2\beta}$ for some $C>0$.
\item\label{item:SG_quasi} For a quasioptimal parameter choice rule $\alpha_*$ and a
regularization method for which $\Delta(\xdag)$ satisfies \eqref{eq:gaps_deltaset} we have
\[
\sup\{\norm{R_{\alpha_*(\delta,\gobs)}\gobs-\xdag}{L^2}:\norm{\gobs-T\xdag}{L^2}\leq \delta\}
= \mathcal{O}\paren{\log(\delta^{-1})^{-\beta}},\qquad \delta\to 0.
\]
\item\label{item:SG_stoch}
$\left(\inf_{\alpha>0}\EW{\norm{R_{\alpha}(T\xdag+\varepsilon W)-\xdag}{L^2}^2}\right)^{1/2}
= \mathcal{O}\paren{\log(\varepsilon^{-1})^{-\beta}},\qquad \varepsilon\to 0.
$
\end{enumerate}
\end{theorem}
\begin{proof}
\emph{(\ref{item:SG_Besov}) $\Leftrightarrow\xdag\in \Xspace_{\kappa}^T$:} 
Theorem \ref{thm:normEquivalence} shows that 
$\xdag\in B^{\beta}_{2,\infty}(\mathbb{S}^2)$ if and only if $\xdag\in\XTkappa$ 
where $\kappa(\alpha) = 2\ln(R) (\ln(\alpha^{-1}))^{-1}(1+o(1))$ as 
$\alpha\rightarrow 0$ since $\Lambda_{\rm SG}(\mu)=\exp(-2\ln(R) \mu^{1/2})(1+o(1))$ 
as $\mu\rightarrow \infty$.


\emph{$\xdag\in \Xspace_{\kappa}^T\Leftrightarrow$ (\ref{item:SG_VSC}) $\Leftrightarrow$ (\ref{item:SG_quasi}) $\Leftrightarrow$ 
(\ref{item:SG_stoch}):} This follows again along the line of the proof of Theorem \ref{thm:backwardHeat}.
Here $\psi_{\kappa^{\beta}} (t)= C\log(t^{-1})^{-2\beta} \rbra{1+o(1)}$ as $t\to 0$. 
\end{proof}

\section{Conclusions}\label{sec:conclusions}
We have described a general strategy for the verification of VSCs. 
For linear operators in Hilbert spaces we have shown via 
a series of equivalence theorems that VSCs are necessary and sufficient 
for certain rates of convergence both for deterministic errors and 
for white noise. For a number of relevant inverse problems 
VSCs with certain index functions are satisfied if and only if 
the solution belongs to some Besov space. 

For other forward operators the set of solutions which satisfies a VSC 
with a (multiple of a) given index function will not be any known function space. 
Nevertheless it is interesting to derive verifiable sufficient conditions 
for VSCs and rates of convergence also for such operators, and we intend to explore 
the potential of our general strategy in such situations in future research. 

Furthermore, our strategy for the verification of VSCs has 
straightforward extensions to Banach spaces, 
general data fidelity and penalty functionals, and it has already 
successfully been applied to nonlinear inverse scattering problems. 
These extensions will be an interesting topic of future research. 
Although VSC are known to be sufficient for certain rates of convergence 
in such general situations, little is known about necessity so far. 
However, we expect that different techniques than those applied in this 
paper will be required for such converse results. 

\section*{Acknowledgement}

We would like to thank Uno H\"amarik und Toomas Raus for helpful discussions 
on Lemma \ref{lem:infSupSwap}. 
Financial support by the German Research Foundation DFG through 
CRC 755, project C09 is gratefully acknowledged. 

\appendix
\section{Spectral source conditions}\label{sec:spectral_sc}
In this appendix we will use the general strategy of \S \ref{sec:strategy} 
to derive variational source conditions from spectral source conditions. 
Compared to the implication (\ref{item:boundedSpectral}) $\Rightarrow$ (\ref{item:VSC}) in Theorem \ref{thm:equiv_vsc_decay}, we can relax the assumption that $t\mapsto \kappa(t)^{2}/t^{1-\mu}$ is decreasing 
for some $\mu\in(0,1)$ by allowing also $\mu=0$. 
Moreover, the proof for spectral source conditions is considerably simpler. 

The result has been known in principle, but previous derivations have 
been indirect via distance functions, did not yield explicit control 
over constants, and already for logarithmic source conditions involved 
quite heavy computations (see \cite{Flemming:12}). 

\begin{proposition}
If $T$ is linear, $\Yspace$ is a Hilbert space, and $\xdag$ satisfies a 
spectral source condition 
\begin{equation*}
\xdag=\varphi(T^*T)w,\qquad \|w\|\leq\rho
\end{equation*}
with an index function $\varphi$ such that $\varphi^2$ is concave, 
then $\xdag$ satisfies the variational source condition \eqref{eq:VSC_innerprod} with 
\begin{equation}\label{eq:psi_spectral}
\psi(\delta^2) = 4\rho^2\varphi\paren{\Theta^{-1}\paren{\frac{\delta}{\rho}}}^2,
\qquad \Theta(t):=\sqrt{t}\varphi(t). 
\end{equation} 
\end{proposition}

\begin{proof}
Let $E_{\projparam}=1_{[0,\projparam]}(T^*T)$ denote the spectral family generated by the operator $T^*T$ and 
set $P_{\projparam}:=I-E_{\projparam}$ for $\projparam>0$. 
Then 
\[
\|(I-P_{\projparam})\xdag\|^2 = \|E(\projparam)\varphi(T^*T)w\|^2 
= \int_0^{\projparam} \varphi(t)^2\,\mathrm{d}\|E_tw\|^2 
\leq \varphi(\projparam)^2\rho^2.
\]
Therefore, \eqref{eq:approximation} holds true with $\kappa(\projparam)=\rho\varphi(\projparam)$. 
Moreover, 
\begin{align*}
\lsp \xdag,P_{\projparam}(\xdag-\sol)\rsp 
&= \lsp w, P_{\projparam}\varphi(T^*T)(\xdag-\sol)\rsp 
\leq \rho \paren{\int_{\projparam}^{\infty} \varphi(t)^2\,\mathrm{d}\|E_{t}(\xdag-\sol)\|^2}^{1/2} \\
&\leq \rho \paren{\sup_{t\geq \projparam}\frac{\varphi(t)^2}{t}
\int_{\rho}^{\infty} t\mathrm{d}\|E_{t}(\xdag-\sol)\|^2}^{1/2} 
\leq \rho \frac{\varphi(\projparam)}{\sqrt{\projparam}} \|T(\xdag-\sol)\|
\end{align*}
where $\sup_{t\geq \projparam}\varphi(t)^2{t}=\varphi(\projparam)^2/\projparam$  since 
$\varphi^2$ is concave and $\varphi(0)=0$. 
Hence,  \eqref{eq:stability} holds true with $\sigma(\projparam) = \rho\varphi(\projparam)/\sqrt{\projparam}$ 
and $C=0$. Therefore, \eqref{eq:VSC_innerprod} holds true with 
\[
\psi(\delta)  = 2\inf_{\projparam>0}\left[\rho^2\varphi(\projparam)^2 
+\frac{\varphi(\projparam)}{\sqrt{\projparam}}\rho \delta \right]
= 2\rho^2 \inf_{\projparam>0}\left[\varphi(\projparam)^2 
+\frac{\varphi(\projparam)}{\sqrt{\projparam}}\frac{\delta}{\rho} \right].
\]
We choose $\projparam= \Theta^{-1}(\delta/\rho)$, i.e.\ $\sqrt{\projparam}\varphi(\projparam)=\delta/\rho$. 
Then $\frac{\varphi(\projparam)}{\sqrt{\projparam}}\frac{\delta}{\rho}
= \varphi(\projparam)^2= \varphi(\Theta^{-1}(\delta/\rho))^2$, so we obtain \eqref{eq:psi_spectral}. 
\end{proof}



\bibliographystyle{abbrv}
\bibliography{besovlit}

\end{document}